\documentclass[a4paper,12pt]{article}
\usepackage{fullpage} 
\usepackage[latin2]{inputenc}
\usepackage[english]{babel}

\usepackage{tikz}

\usepackage{dsfont}

\usepackage{amsmath}
\usepackage{amsthm}
\usepackage{amssymb}
\usepackage{amsfonts}

\usepackage{fancyhdr}

\usepackage{framed}

\usepackage{graphics}
\usepackage{graphicx}

\usepackage{t1enc}

\DeclareMathOperator*{\esssup}{ess\,sup}

\DeclareMathOperator{\sgn}{ sgn}

\theoremstyle{definition}
\newtheorem{defi}{Definition}

\theoremstyle{plain}
\newtheorem{tetel}{Theorem}
\newtheorem{allitas}{Statement}
\newtheorem{lemma}{Lemma}

\theoremstyle{remark}
\newtheorem{megj}{Remark}
\newtheorem{lepes}{Step}
\newtheorem{lepess}{Step}

\title{Weighted $\gamma$-$K$-functional and $\gamma$-Modulus of Smoothness on the Semiaxis}
\author{Zolt\'an Mark\'o\footnote{e-mail: marzol@math.bme.hu} \\
Department of Analysis, \\
Budapest University of Technology and Economics, \\
H-1521 Budapest, POB 91, Hungary}

\begin{document}

\maketitle

\begin{abstract}

\par

In this paper we investigate the $\gamma$-relative differentiation by the motivation of amending the order of the weighted polynomial approximation on the semiaxis for certain functions. With the help of this we give some definitions of generalized Sobolev spaces, $K$-functionals and moduli of smoothness. We prove theorems for estimating these things with each other, in the case of first order we prove equivalence. We remark some possible applications and other generalizations too.

\end{abstract}


\section{Introduction}

Polynomial approximation is a traditional area in approximation theory, in the course of this one might ask the question of the order of approximation. That is, if $X$ is a Banach space of functions with finite domains of $\mathds{R}$, and $P_n$ is the space of polynomials of degree at most $n$ on $X$, then the error of the best polynomial approximation for a given $f \in X$ function is:
\begin{equation*} 
	E_n(f) := \inf_{p_n \in P_n} \| f - p_n \|.
\end{equation*}
Naturally the order of the error depends on the properties of the function. One of the basic ideas is that the order of the approximatibility can be connected to the smoothness properties of $f$, which can be characterized with the modulus of smoothness. This type of theorems are the Jackson-, Bernstein-, Stechkin-type theorems and saturation theorems (in the simplest trigonometric case see \cite{devore1993}, \cite{Jackson1930} and \cite{Stechkin1951}).

The $K$-functional is a useful tool in several areas of mathematics, e. g. functional analysis, harmonic analysis and the theory of ordinary differential equations. In 1969 J. Peetre \cite{Peetre1969} drew attention to the application of approximation theory, he used the $K$-functional for interpolating operators.

The method developed in the last 40 years is applicable to generalize the classical Jackson-, Bernstein-, etc. type theorems for several Banach spaces: the technique of the proof is based on the equivalence of the $K$-functional and the modulus of smoothness (this type of proofs can be seen in \cite{ditzian1987}, \cite{Mhaskar1996} books). Once the classical results were there, several further generalizations were proved in the topic: e. g. Dekel \cite{Dekel2010} proved equivalence over convex domains, over the $n$-dimensional sphere Rustamov \cite{Rustamov1992} proved theorems.

In the last decades polynomial approximation was investigated in weighted spaces, so the introduction of the weighted moduli and $K$-functionals was inevitable (for example \cite{Ditzian1997}, \cite{ditzian1987}, \cite{Ky1993}).

One of the most interesting case (which involves the problem of finite and infinite) is the $(0,\infty)$ interval equipped with the $w_{\alpha}(x) = x^{\alpha} e^{-x}$ Laguerre weight, for which De Bonis, Mastroianni and Viggiano \cite{DeBonis2002} proved equivalence theorems. 

The goal of this paper is to generalize these results with the help of $\gamma$-relative differentiation, which we will define in section 2. 

With the introduction of $\gamma$-relative derivation our topic is connected to the following problem: Faber proved that one can not achieve Jackson-order approximation with projection type operators for continuous functions in uniform metric. Since Jackson-type theorems use global smoothness properties of the functions, it is a natural question, whether can one amend the order of approximation, if the local properties of the function are better. One of the possible solutions is that we paste piecewise very smooth functions with a less smooth way, in this case the order of apprimation gets indeed better (\cite{Horv'ath1996}, \cite{Li1995}, \cite{Mastroianni1993}, and \cite{Mastroianni1995}). To achieve this $\gamma$-relative differentiation is a good option, because with by choosing a suitable $\gamma$ the $\gamma$-relative differentiable functions will be just elements of such a function space. These above motivate the introducing of $\gamma$-$K$-functional and $\gamma$-modulus of smoothness, with the characterization of their connection we can get new results in the topic.

In the second section we define the $\gamma$-relative differentiation and give some statement, we construct a suitable $\gamma$. In the third section we give the definitions of some function spaces and the $\gamma$-$K$-functional, $\gamma$-modulus of smoothness. Here we state the main results of this paper, and we draw attention to some other possible definitions and theorems as well. In the last section we prove all our theorems.


\section{The $\gamma$-relative Differentiation} \label{478a}

In articles \cite{amat2007} and \cite{yanjie2009} the authors dealt with the interpolation of continuous piecewise smooth functions, with the generalization of differentiation. That is the so called $\gamma$-relative differentiation, with the help of which some functions will be $\gamma$-relative differentiable even if they were not differentiable in the classical sense.

\begin{defi} \label{a24}
	\cite{amat2007}, \cite{yanjie2009} Let $I$ be the neighborhood of $x_0 \in \mathds{R}$, and let $\gamma : I \rightarrow \mathds{R}$ continuous and strictly increasing function. We say that the $f : I \rightarrow \mathds{R}$ function is $\gamma$-relative differentiable in $x_0$, if the limit
	\begin{equation*} 
		\lim_{h \rightarrow 0} \frac{f(x_0 + h) - f(x_0)}{\gamma(x_0 + h) - \gamma(x_0)}
	\end{equation*}
	exists and it is finite. Then we use the notation $f'_{\gamma}(x_0)$ for this limit, and we call it the $\gamma$-relative derivation of $f$ in $x_0$. If the $f : I \rightarrow \mathds{R}$ function is $\gamma$-relative differentiable in all points of $I$, then we call $f$ $\gamma$-relative differentiable, and the $f'_{\gamma} : I \rightarrow \mathds{R}$ function is the $\gamma$-relative derivation function of $f$.
\end{defi}


From now on we suppose that $D_f \subset D_{\gamma}$. 
We can define the higher order derivatives recursively: $f_{\gamma}^{(1)} = f'_{\gamma}$, $\ldots$, $f_{\gamma}^{(n)} = (f_{\gamma}^{(n-1)})'_{\gamma}$.

It is easy to check that many of the properties of the classical differentiation hold for $\gamma$-relative differentiation, e. g. the statements for the differentiation of the sum, product, quotient stay also true. The Leibniz rule is true as well:

\begin{allitas}[generalization of the Leibniz rule] \label{140}
	Let $\gamma$ strictly increasing continuous function in the neighborhood of $x_0$. If $f$ and $g$ are $n$-times $\gamma$-relative differentiable in $x_0$ ($n \geq 1$), then $f \cdot g$ is differentiable too, and
	\begin{equation*} 
		(fg)^{(r)}_{\gamma}(x_0) = \sum_{k=0}^{n} \binom{n}{k} f^{(k)}_{\gamma}(x_0) \cdot g^{(n-k)}_{\gamma}(x_0).
	\end{equation*}
\end{allitas}



The following statements will play a major role later on, in section \ref{566} we will prove them.

\begin{allitas} \label{337}
	Let $\gamma(x) = ax + b$, where $a,b \in \mathds{R}$, $a \neq 0$. Then $f$ is $r$-times $\gamma$-relative differentiable on $I$ $\Leftrightarrow$ $f$ is $r$-times differentiable on $I$, and for all $x \in I$
	\begin{equation*} 
		f^{(r)}_{\gamma}(x) = \frac{1}{a^r} f^{(r)}(x).
	\end{equation*} 
\end{allitas}

\begin{allitas} \label{17}
	$f$ is $r$-times $\gamma$-relative differentiable on $I$ $\Leftrightarrow$ $f \circ \gamma^{-1}$ is $r$-times differentiable on $\gamma(I)$. Then, if $y \in \gamma(I)$:
		$(f_{\gamma}^{(r)} \circ \gamma^{-1})(y) = (f \circ \gamma^{-1})^{(r)}(y)$.
\end{allitas}

\begin{allitas} \label{25}
	$f \circ \gamma$ is $r$-times $\gamma$-relative differentiable on $I$ $\Leftrightarrow$ $f$ is $r$-times differentiable on $\gamma(I)$. Then, if $x \in I$:
		$(f \circ \gamma)^{(r)}_{\gamma}(x) = (f^{(r)} \circ \gamma)(x)$.
\end{allitas}

\begin{defi} \label{561}
	If $p_n $ is a polynomial of degree at most $n$, $\gamma : I \rightarrow \mathds{R}$ is a continuous, strictly increasing function, then $p_n \circ \gamma : I \rightarrow \mathds{R}$ function is a $\gamma$-polynomial of degree $n$.
\end{defi}

By the  statement \ref{25}, if $p_m$ is a polynomial of degree at most $m$, then
	\begin{equation} \label{452}
		(p_m \circ \gamma)^{(r)}_{\gamma} = 0
	\end{equation}
if $r \geq m+1$.

\begin{allitas}[partial integration] \label{22}
	Let us suppose that $f$ and $g$ are real valued functions on $[a,b]$, they are $\gamma$-relative differentiable on $(a,b)$, and $f'_{\gamma}$, $g'_{\gamma}$ are integrable on $[a,b]$. Then
	\begin{equation} \label{23}
		\int_{a}^{b} f'_{\gamma}(x) g(x) \, \mathrm{d} \gamma(x) = (fg)(b) - (fg)(a) - \int_{a}^{b} f(x) g'_{\gamma}(x) \, \mathrm{d} \gamma(x).
	\end{equation} 
\end{allitas}
In the case of the $\gamma$-relative differentiable functions the suitable Taylor $\gamma$-polynomial can be define, with similar error term.

\begin{allitas}[generalized Taylor-polynomial] \label{28}
	\cite{yanjie2009} Let us suppose that $f$ is an at least $n+1$-times $\gamma$-relative differentiable function in a neighborhood of point $x_0 \in I$. Let
	\begin{equation*} 
		T_{\gamma,n}(x) = \sum_{k=0}^{n} \frac{f_{\gamma}^{(k)}(x_0)}{k!}(\gamma(x) - \gamma(x_0))^k.
	\end{equation*}
	Then there exists a $K$ neighborhood of $x_0$, that if $x \in K$, then $f(x) = T_{\gamma,n}(x) + R_{\gamma,n}(x)$, where
	\begin{equation*} 
		R_{\gamma,n}(x) = \frac{f_{\gamma}^{(n+1)}(\xi)}{(n+1)!} (\gamma(x) - \gamma(x_0))^{n+1}, \ \ \ \xi \in K.
	\end{equation*}
\end{allitas}

We can write the error term with the help of an integral:

\begin{allitas} \label{31}
	Let us suppose that $f$ is a function $\gamma$-relative differentiable at least $r$-times in a $K$ neighborhood of point $x_0 \in I$. Then, if $x \in K$:
	\begin{equation*} 
		f(x) - T_{\gamma,r-1}(x) = \int_{x_0}^{x} f_{\gamma}^{(r)}(t) \frac{(\gamma(x) - \gamma(t))^{r-1}}{(r-1)!} \, \mathrm{d} \gamma(t).
	\end{equation*}
\end{allitas}

Let $\{ a_1,a_2,\ldots,a_N \} \subset \mathds{R}^+$ be a system of points such that $0 < a_1 < a_2 < \ldots < a_N < \infty$. From articles \cite{amat2007} and \cite{yanjie2009} we construct a $\gamma$ function which improves the smoothness properties of a function of these points. To do so, we give a function with relatively simple structure, which is strictly increasing, non-negative and differentiable on $\mathds{R}^+ \setminus \{ a_1,a_2,\ldots,a_N \}$, and
	\begin{enumerate}
		\item $\gamma'(a_k) = \infty$, if $k=1,\ldots,N$.
	\end{enumerate}
It is assumed that
	\begin{enumerate}
		\setcounter{enumi}{1}
		\item $\gamma(0) = 0$, and
		\item $\gamma$ is linear, if $a_N + 1 \leq x$,
	\end{enumerate}
	because for big $x$-es the behavior of the weight function will be dominant, so linearity is a natural assumption. We can get a function with the conditions above in several ways, for example let $\gamma$ on $[t_k,t_{k+1}]$ be the corresponding transformation of the arcus sine function.
	
	An another way to constuct a right function is to write $\gamma$ in the form $\gamma := \sum_{k=1}^{N} \gamma_k$, where $\gamma_k$ is a strictly increasing, non-negative, continuous function, which is differentiable on $\mathds{R}^+ \setminus \{ a_k \}$ and $\gamma'(a_k) = \infty$. For our goal this definition suits better, more specifically let $0 < \beta_k < 1$,

	\begin{equation*} 
		\gamma_k(x) := |x - a_k|^{\beta_k} \sgn(x-a_k)
	\end{equation*}
	and
	\begin{equation*} 
		\gamma_0(x) = \sum_{k=1}^{N} |x-a_k|^{\beta_k} \mathrm{sgn} \, (x-a_k) + \sum_{k=1}^{N} a_k^{\beta_k}.
	\end{equation*}
	Let us define then $\gamma$ function as follows:
	\begin{equation} \label{203}
		\gamma(x) := \left\{
			\begin{array}{rl}
				\gamma_0(x), & \text{if } 0 \leq x \leq a_N+1, \\
				\mathcal{C}_1x+\mathcal{C}_2, & \text{if } a_{N}+1 \leq x.
			\end{array}
		\right.
	\end{equation}
	We choose the $\mathcal{C}_1$ and $\mathcal{C}_2$ factors to make the given $\gamma$ be continuous, which makes it continuously differentiable as well:
	\begin{align}
		\mathcal{C}_1 & := \sum_{k=1}^{N} \beta_k (a_N + 1 - a_k)^{\beta_k - 1}, \label{204} \\
		\mathcal{C}_2 & := \sum_{k=1}^{N} (a_N+1-a_k)^{\beta_k} - \sum_{k=1}^{N} \beta_k (a_{N}+1-a_k)^{\beta_k - 1}(a_N + 1) + \sum_{k=1}^{n} a_k^{\beta_k}. \label{205}
	\end{align}

This $\gamma$ will fulfill all of the properties above, from this point we investigate the $\gamma$ given by (\ref{203}).

Later it will be important for us, that the inverse of $\gamma$ be $r$-times differentiable on $(0,\gamma(a_N+1))$ and on $(\gamma(a_N+1),\infty)$. In the second interval the inverse function is linear too, so it is infinitely many times differentiable. On the $(0,\gamma(a_N+1))$ interval the previous $\gamma$ will not be necessarily $r$-times differentiable, but the following statement is true:

\begin{allitas} \label{121}
	Let us regard the $\gamma$ function which is given by formula (\ref{203}). If $0 < \beta_k < \frac{1}{r}$ for all $k=1,\ldots,N$, then $\gamma^{-1}$ is $r$-times differentiable on $(0,\gamma(a_N+1))$. 
\end{allitas}

We suppose from here that $0 < \beta_k < \frac{1}{r}$ for all $1 \leq k \leq N$.

\begin{allitas} \label{197}
	If $0 \leq x \leq a_1$, then
	\begin{equation} \label{198}
		\sum_{k=1}^{N} \beta_k a_k^{\beta_k - 1} \cdot x \leq \gamma(x) \leq \frac{\gamma(a_1)}{a_1} \cdot x.
	\end{equation}
\end{allitas}

It follows from this last statement reflecting the graph of $\gamma$ to the $y=x$ line, that if $0 \leq x \leq \gamma(a_1)$, then
\begin{equation} \label{200}
	\frac{a_1}{\gamma(a_1)} \cdot x \leq \gamma^{-1}(x) \leq \frac{1}{\sum_{k=1}^{N} \beta_k a_k^{\beta_k - 1}} \cdot x.
\end{equation}


\section{Weighted $\gamma$-$K$-functional and $\gamma$-Modulus of Smoothness}

If $x \geq 0$, $\alpha > - \frac{1}{p}$, $1 \leq p < \infty$, let $w_{\alpha}(x) = x^{\alpha} e^{-x}$, a Laguerre weight. If $0 \leq a < b \leq \infty$, then let
	\begin{equation*} 
		L_{w_{\alpha}}^p(a,b) := \left\{ f : (a,b) \rightarrow \mathds{R} \mid \| f w_{\alpha} \|_{L^p(a,b)}^p := \int_a^b |f(x) w_{\alpha}(x)|^p \, \mathrm{d} x < \infty \right\}.
	\end{equation*}
In the case $a= 0$, $b = \infty$ we use the more simple $\| f w_{\alpha} \|_{L^p}$ notation instead of the previous one, and we refer to the function space with the $L_{w_{\alpha}}^p$ notation.

Furthermore, if $\alpha \geq 0$, then let
	\begin{equation*} 
		L^{\infty}_{w_{\alpha}}(a,b) := \left\{ f : (a,b) \rightarrow \mathds{R} \mid \| f w_{\alpha} \|_{L^{\infty}(a,b)} := \esssup_{0 \leq x} |f(x) w_{\alpha}(x)| < \infty \right\}.
	\end{equation*}
In the case $a= 0$, $b = \infty$ we use the more simple $\| f w_{\alpha} \|_{L^{\infty}}$ notation, and we note the function space with $L_{w_{\alpha}}^{\infty}$.

We generalize the weighted Sobolev spaces from \cite{DeBonis2002}: let $1 \leq p \leq \infty$, $r \geq 1$ integer, and with $\varphi(x) = \sqrt{x}$, $I \subset \mathds{R}$ let 
	\begin{equation} \label{214}
			W_{\gamma}^{r,p}(I) := W_{\gamma}^{r,p}(I,w_{\alpha}) = \left\{ f \in L^p_{w_{\alpha}}(I) : 
			\| f^{(r)}_{\gamma} \varphi^r w_{\alpha} \|_{L^p(I)} < \infty \right\},
	\end{equation}
the correspoding half-norm and norm are
	\begin{equation*} 
		|f|_{W_{\gamma}^{r,p}(I)} = \| f^{(r)}_{\gamma} \varphi^{r} w_{\alpha} \|_{L^p(I)}, \ \ \  \| f \|_{W_{\gamma}^{r,p}(I)} := \| f w_{\alpha} \|_{L^p(I)} + \| f^{(r)}_{\gamma} \varphi^{r} w_{\alpha} \|_{L^p(I)}.
	\end{equation*}
	If $I = (0,\infty)$, then we use the notation $W_{\gamma}^{r,p} = W_{\gamma}^{r,p}(w_{\alpha})$.

\begin{defi} \label{568}
In the generalized case the weighted $\gamma$-$K$-functional for the $L^p_{w_{\alpha}}$ and $W_{\gamma}^{r,p}$ spaces is
	\begin{equation} \label{13}
		K_{\gamma,r,\varphi}(f,t^r,L^p_{w_{\alpha}},W_{\gamma}^{r,p}) := \inf_{g \in W_{\gamma}^{r,p}} \left\{ \| (f-g) w_{\alpha} \|_{L^p} + t^r \| g^{(r)}_{\gamma} \varphi^r w_{\alpha} \|_{L^p} \right\},
	\end{equation}
where $0 < t \leq 1$, $1 \leq p \leq \infty$.
\end{defi}

Due to the nature of the problem it turns up, that it is practical to regard the above notions on a suitable finite $I \subset \mathds{R}^+$ interval instead of the representation on $(0,\infty)$, and to investigate the domains close to zero and close to infinity separately.

\begin{defi} \label{569}
Let
\begin{equation} \label{479}
	I_{rh,\gamma} = \left[4 A_1 r^2 (\gamma^{-1}(h))^2, \frac{A_2}{(\gamma^{-1}(h))^{2}} \right],
\end{equation}
where $r > 0$ is an integer, $A_1,A_2 > 0$ are real numbers (we will take another conditions for $A_1$ and $A_2$ in the proofs). Let us define the resticted $K$-functional as follows:
	\begin{multline} \label{14}
		\widetilde{K}_{\gamma,r,\varphi}(f,t^r,L^p_{w_{\alpha}},W_{\gamma}^{r,p}) \\
		:= \sup_{0 < h \leq t} \inf_{g \in W_{\gamma}^{r,p}(I_{rh,\gamma})} \left\{ \| (f-g) w_{\alpha} \|_{L^p(I_{rh,\gamma})} + h^r \| g^{(r)}_{\gamma} \varphi^r w_{\alpha} \|_{L^p(I_{rh,\gamma})} \right\}.
	\end{multline}
 \end{defi}

To the valid definiton of $I_{rh,\gamma}$ we must suppose that $h \leq \gamma \left( \sqrt[4]{\frac{A_2}{A_1}} \frac{1}{\sqrt{2r}} \right)$. We do not sign $A_1,A_2$, if they play an important role, we highlight them.

In the book \cite{ditzian1987} of Ditzian and Totik there is a definition of the main part of the $\varphi$-modulus of smoothness, in analogy with this let us consider the following definiton:

\begin{defi} \label{570}
The main part of the $\gamma$-modulus of smoothness is
	\begin{equation} \label{48}
		\Omega_{\gamma,\varphi}^{r}(f,t)_{w_{\alpha},p} := \sup_{0 < h \leq t} \left\| w_{\alpha} \overrightarrow{\Delta}_{\gamma,\varphi,h}^{r} f  \right\|_{L^p(I_{rh,\gamma})},
	\end{equation}
where \begin{equation*} 
		\overrightarrow{\Delta}_{\gamma,\varphi,h}^{r}f(x) = \sum_{k=0}^{r} (-1)^{r-k} \binom{r}{k} f(x + k \gamma^{-1}(h) \varphi(x)),
	\end{equation*}
$\varphi(x) = \sqrt{x}$.
\end{defi}

\begin{defi} \label{571}
 The complete modulus of smoothness is
	\begin{multline} \label{480}
	\omega_{\gamma,\varphi}^{r}(f,t)_{w_{\alpha},p} := \Omega_{\gamma,\varphi}^{r}(f,t)_{w_{\alpha},p} + \inf_{p \in P_{r-1}} \| w_{\alpha} (f - p \circ \gamma) \|_{L^p(0,4A_1r^2 (\gamma^{-1}(t))^2)} \\
	+ \inf_{q \in P_{r-1}} \| w_{\alpha} (f - q \circ \gamma) \|_{L^p\left( \frac{A_2}{(\gamma^{-1}(t))^2}, \infty \right)},
\end{multline}
where $0 < t \leq 1$, $1 \leq p \leq \infty$, and $P_n$ is the set of algebraic polynomials of degree at most $n$.
\end{defi}


In the following we compare the $\gamma$-$K$-functional and $\gamma$-modulus of smoothness above, and with this we prove the analogous of theorem 2.1. in \cite{DeBonis2002}.

The (2.6) statement of this theorem corresponds to the following theorem:

\begin{tetel} \label{50}
	Let $f \in L^p_{w_{\alpha}}$, $\alpha \geq 0$, $1 \leq p \leq \infty$, $r \in \mathds{N}^+$, and let $\widetilde{K}_{\gamma,r,\varphi}$ and $\Omega_{\varphi,\gamma}^{r}(f,t)_{w_{\alpha},p}$ be given by (\ref{14}) and (\ref{48}). Then for all $1 \leq p \leq \infty$
	\begin{equation} \label{51}
		\widetilde{K}_{\gamma,r,\varphi}(f,t^r,L^p_{w_{\alpha}},W_{\gamma}^{r,p})	\leq C \sum_{n=1}^{r} t^{r-n} \Omega_{\gamma,\varphi}^{n} (f,t)_{w_{\alpha},p}
	\end{equation}
	and
	\begin{equation} \label{391}
		\Omega_{\gamma,\varphi}^{r}(f,t)_{w_{\alpha},p} \leq C \cdot \widetilde{K}_{\gamma,r,\varphi}(f,t^r,L^p_{w_{\alpha}},W_{\gamma}^{r,p})
	\end{equation}
	hold for all $t < T$, where $T$ and the constants in the estimates are independent of $f$ and $t$.
\end{tetel}

The corresponding theorem for the statement (2.7) of theorem 2.1 in \cite{DeBonis2002} is the following:

\begin{tetel} \label{481}
	Let $f \in L^p_{w_{\alpha}}$, $\alpha \geq 0$, $1 \leq p \leq \infty$, $r \in \mathds{N}^+$, and let $K_{\gamma,r,\varphi}$ and $\omega_{\varphi,\gamma}^{r}(f,t)_{w_{\alpha},p}$ be given by (\ref{13}) and (\ref{480}). Then for all $1 \leq p \leq \infty$
	\begin{equation} \label{482}
		K_{\gamma,r,\varphi}(f,t^r,L^p_{w_{\alpha}},W_{\gamma}^{r,p})	\leq C \sum_{n=1}^{r} t^{r-n} \omega_{\gamma,\varphi}^{n} (f,t)_{w_{\alpha},p}
	\end{equation}
	and
	\begin{equation} \label{483}
		\omega_{\gamma,\varphi}^{r}(f,t)_{w_{\alpha},p} \leq C \cdot K_{\gamma,r,\varphi}(f,t^r,L^p_{w_{\alpha}},W_{\gamma}^{r,p})
	\end{equation}
	hold for all $t < T$, where $T$ and the constants in the estimates are independent of $f$ and $t$.
\end{tetel}

\begin{megj} \label{559}
	The proof of lemma \ref{37} and theorem \ref{50} are valid too when $-\frac{1}{p} < \alpha < 0$ holds, however, in step 2 of the proof of theorem \ref{481} we used that $\alpha \geq 0$.
\end{megj}

\begin{megj} \label{586}
	We say that $A \sim B$ ($A$ and $B$ are equivalent), iff there exists a positive constant $C$ independent of $A$ and $B$ such that $C^{-1} \leq \frac{A}{B} \leq C$. In the case $r=1$ we proved equivalence: by theorems \ref{50} and \ref{481} the following statement is true:
\end{megj}

\begin{tetel} \label{550}
	Let $f \in L^p_{w_{\alpha}}$, $\alpha \geq 0$, $1 \leq p \leq \infty$, and let $\widetilde{K}_{\gamma,1,\varphi}$, $\Omega_{\varphi,\gamma}^{1}(f,t)_{w_{\alpha},p}$, $K_{\gamma,1,\varphi}$, and $\omega_{\varphi,\gamma}^{1}(f,t)_{w_{\alpha},p}$ be given by (\ref{14}), (\ref{48}), (\ref{13}) and (\ref{480}) (the case $r=1$). Then for all $p$
	\begin{align*}
		\widetilde{K}_{\gamma,1,\varphi}(f,t,L^p_{w_{\alpha}},W_{\gamma}^{1,p}) & \sim \Omega_{\gamma,\varphi}^{1}(f,t)_{w_{\alpha},p} \\ 
		K_{\gamma,1,\varphi}(f,t,L^p_{w_{\alpha}},W_{\gamma}^{1,p}) & \sim \omega_{\gamma,\varphi}^{1}(f,t)_{w_{\alpha},p} 
	\end{align*}
	holds for all $t < T$,  where $T$ and the constants in $\sim$ are independent of $f$ and $t$.
\end{tetel}

Finally we note some statements, which can be proved an analogous proof of the theorems above.

\begin{megj} \label{583}
	Let us consider instead of (\ref{214}) the following function space:
	\begin{equation*} 
		\widehat{W}_{\gamma}^{r,p}(I) := \widehat{W}_{\gamma}^{r,p}(I,w_{\alpha}) := \{ f \in L^p_{w_{\alpha}}(I): \| (f^{(r-1)})'_{\gamma} \varphi^r w_{\alpha} \|_{L^p(I)} < \infty \}.
	\end{equation*}
	Since the restricted $K$-functional and the main part of modulus of smoothness given by the classical differentation are equivalent \cite{DeBonis2002}, furthermore by theorem \ref{550} we have equivalence for the restricted $\gamma$-$K$-functional and the main part of $\gamma$-modulus of smoothness in the case $r=1$, combining the proofs of \cite{DeBonis2002} and theorem \ref{50} and \ref{481} we can get for arbitrary $r \geq 1$ equivalence theorems for the main part of the $\gamma$-modulus of smoothness given by (\ref{48}) and for the following restricted $K$-functional:
	\begin{multline*} 
		\widetilde{K}_{r,\varphi} \left(f,t^r,L^p_{w_{\alpha}},\widehat{W}_{\gamma}^{r,p} \right) \\
		:= \sup_{0 < h \leq t} \inf_{g \in \widehat{W}_{\gamma}^{r,p}(I_{rh},\gamma)} \left\{ \| (f-g) w_{\alpha} \|_{L^p(I_{rh,\gamma})} + h^r \| (g^{(r-1)})'_{\gamma} \varphi^r w_{\alpha} \|_{L^p(I_{rh,\gamma})} \right\},
	\end{multline*}
	where $I_{rh,\gamma}$ is given by (\ref{479}). As above, we get the following theorem:
\end{megj}

\begin{tetel} \label{588}
	Let $f \in L^p_{w_{\alpha}}$, $\alpha \geq 0$, $1 \leq p \leq \infty$, $r \in \mathds{N}^+$. Then
	\begin{equation*} 
		\Omega_{\gamma, \varphi}^r(f,t)_{w_{\alpha},p} \sim \widetilde{K}_{r,\varphi} \left( f,t^r,L^p_{w_{\alpha}},\widehat{W}^{r,p} \right) 
	\end{equation*}
	holds for all $t < T$, where $T$ and the constants in $\sim$ are independent of $f$ and $t$.
\end{tetel}

\begin{megj} \label{454}
	The statement (\ref{51}) of theorem \ref{50} remains true even if instead of $L^p_{w_{\alpha}}(I)$ and $W_{\gamma}^{r,p}(I)$ spaces given by the norm with Lebesgue measure we consider the corresponding $L^p_{w_{\alpha}}(\lambda_{\gamma},I)$ and $W_{\gamma}^{r,p}(\lambda_{\gamma},I)$ spaces given by the norm with the $\lambda_{\gamma}$ Lebesgue--Stieltjes measure. In this case the restricted $\gamma$-$K$-functional is the following:
	\begin{multline*} 
		\widetilde{K}_{\gamma,r,\varphi}(f,t^r,L^p_{w_{\alpha}}(\lambda_{\gamma}),W_{\gamma}^{r,p}(\lambda_{\gamma})) \\
		:= \sup_{0 < h \leq t} \inf_{g \in W_{\gamma}^{r,p}(\lambda_{\gamma},I_{rh,\gamma})} \left\{ \| (f-g) w_{\alpha} \|_{L^p(\lambda_{\gamma},I_{rh,\gamma})} + h^r \| g^{(r)}_{\gamma} \varphi^r w_{\alpha} \|_{L^p(\lambda_{\gamma},I_{rh,\gamma})} \right\},
	\end{multline*}
	where $I_{rh,\gamma}$ is given by (\ref{479}). The main part of the $\gamma$-modulus of smoothness is given by (\ref{48}), where now we integrate by the $\lambda_{\gamma}$ measure. Then the following statement is true:
\end{megj}

\begin{tetel} \label{589}
		Let $f \in L^p_{w_{\alpha}}(\lambda_{\gamma})$, $\alpha \geq 0$, $1 \leq p \leq \infty$, $r \in \mathds{N}^+$. Then for all $1 \leq p \leq \infty$
	\begin{equation*} 
		\widetilde{K}_{\gamma,r,\varphi}(f,t^r,L^p_{w_{\alpha}}(\lambda_{\gamma}),W_{\gamma}^{r,p}(\lambda_{\gamma}))	\leq C \sum_{n=1}^{r} t^{r-n} \Omega_{\gamma,\varphi}^{n} (f,t)_{w_{\alpha},p}
	\end{equation*}
	holds for all $t < T$, where $T$and the constants in $\sim$ are independent of $f$ and $t$.
	\end{tetel}

For this we need the equivalent of lemma \ref{359} with the $\gamma$-relative derivative, which is the following:

\begin{lemma} \label{590}
	Let $0 \leq a \leq b$. If $f, f^{(r)}_{\gamma} \in L^p((a,b),\lambda_{\gamma})$, $1 \leq p \leq \infty$, then for all $0 \leq k \leq r$
	\begin{multline*} 
		\| f^{(k)}_{\gamma} \|_{L^p((a,b),\lambda_{\gamma})} \leq M(r,k) \| f \|_{L^p((a,b),\lambda_{\gamma})} (\gamma(b)-\gamma(a))^{-k} \\
		+ M(r,k) \| f^{(r)}_{\gamma} \|_{L^p((a,b),\lambda_{\gamma})} (\gamma(b)-\gamma(a))^{r-k}
	\end{multline*}
	holds, where $M$ is independent of $p$ and $[a,b]$.
\end{lemma}
	
	We can prove this statement similarly to the classical result in \cite{Ditzian1980}.

\section{Proofs} \label{566}

We prove first the statements for the $\gamma$-relative differentiation. The proof of statement \ref{337} is a simple mathematical induction for $r$.



\begin{proof}[Proof of statement \ref{17}]
	We use mathematical induction. For $r=1$: let $y_0 \in \gamma(I)$. If $f$ is $\gamma$-relative differentiable on $I$, then
	\begin{equation*} 
		(f'_{\gamma} \circ \gamma^{-1})(y_0) 
		= \lim_{y \rightarrow y_0} \frac{f(\gamma^{-1}(y)) - f(\gamma^{-1}(y_0))}{y - y_0} = (f \circ \gamma^{-1})'(y_0).
	\end{equation*}
	We can do this in the opposite direction, so the statement is true for $r=1$. In the transformations we introduce the $\gamma(x) =: y$ variable.
	
	Let us suppose now, that the statement is true for $1 \leq n < r$. Then for $n+1$, with the induction hypothesis:
	\begin{multline*} 
		(f^{(n+1)}_{\gamma} \circ \gamma^{-1})(y_0) = (f^{(n)}_{\gamma})'_{\gamma}(\gamma^{-1}(y_0)) 
		= \lim_{y \rightarrow y_0} \frac{f^{(n)}_{\gamma}(\gamma^{-1}(y)) - f^{(n)}_{\gamma}(\gamma^{-1}(y_0))}{\gamma(\gamma^{-1}(y)) - y_0} \\
		= \lim_{y \rightarrow y_0} \frac{(f \circ \gamma^{-1})^{(n)}(y) - (f \circ \gamma^{-1})^{(n)}(y_0)}{y - y_0} 
		= (f \circ \gamma^{-1})^{(n+1)}(y_0).
	\end{multline*}
	Our steps are reversible, so we proved the statement with the principle of mathematical induction.
\end{proof}

The proof of statement \ref{25} is similar to the proof of statement \ref{17} above, it is easy using mathematical induction, we do not prove it here.
	

\begin{proof}[Proof of statement \ref{22}]
	Let $\gamma(x) =: y$, use statement \ref{17}, and integrate partially on the classical way. Finally with the $x := \gamma^{-1}(y)$ substitution we get (\ref{23}).
\end{proof}

\begin{proof}[Proof of statement \ref{31}]
	By statement \ref{22} we integrate partially, and by using statement \ref{25}, we get that
	\begin{multline} \label{35} 
		\int_{x_0}^{x} (f_{\gamma}^{(r-1)}(t))'_{\gamma} \frac{(\gamma(x) - \gamma(t))^{r-1}}{(r-1)!} \, \mathrm{d} \gamma(t) \\
		= \left[ f_{\gamma}^{(r-1)}(t) \frac{(\gamma(x) - \gamma(t))^{r-1}}{(r-1)!} \right]_{x_0}^{x} - \int_{x_0}^{x} f_{\gamma}^{(r-1)}(t) \left( \frac{(\gamma(x) - \gamma(t))^{r-1}}{(r-1)!} \right)_{\gamma}' \, \mathrm{d} \gamma(t) \\
		= -f_{\gamma}^{(r-1)}(x_0) \frac{(\gamma(x) - \gamma(x_0))^{r-1}}{(r-1)!} + \int_{x_0}^{x} f_{\gamma}^{r-1}(t) \frac{(\gamma(x) - \gamma(t))^{r-2}}{(r-2)!} \, \mathrm{d} \gamma(t)
	\end{multline}
	After partially integrating $r-2$ times we get that (\ref{35}) is equal with
	\begin{multline*} 
		= - f_{\gamma}^{(r-1)}(x_0) \frac{(\gamma(x) - \gamma(x_0))^{r-1}}{(r-1)!} - f_{\gamma}^{(r-2)}(x_0) \frac{(\gamma(x) - \gamma(x_0))^{r-2}}{(r-2)!} - \ldots \\
		- f_{\gamma}'(x_0) \frac{\gamma(x) - \gamma(x_0)}{1!}	+ \int_{x_0}^{x} f_{\gamma}'(t) \cdot 1 \, \mathrm{d}\gamma(t) = f(x) - \sum_{k=0}^{r-1} f_{\gamma}^{(k)}(x_0) \frac{(\gamma(x) - \gamma(x_0))^k}{k!},
	\end{multline*}
	which is the statement we wanted to prove.
\end{proof}


To prove statement \ref{121}, we use the following statement by Fa\`{a} di Bruno for the higher order derivatives of a composition function (its proof is for example in article \cite{Roman1980}). We use the Fa\`{a} di Bruno formula later in several proofs too.

\begin{allitas}[Fa\`{a} di Bruno formula] \label{562}
	\cite{Roman1980} If $f: D_f \subset \mathds{R} \rightarrow \mathds{R}$ and $g : D_g \subset \mathds{R} \rightarrow \mathds{R}$ are functions, that all the derivatives of which in (\ref{563}) are well defined, then for all $1 \leq r \in \mathds{N}$
	\begin{equation} \label{563}
		(f \circ g)^{(r)}(x) = \sum_{l=1}^{r} f^{(l)}(g(x)) \cdot B_{r,l}(g'(x),g''(x), \ldots, g^{(r-l+1)}(x))
	\end{equation}
	holds, where
	\begin{equation} \label{564}
		B_{r,l}(x_1,x_2, \ldots, x_{r-l+1}) 
		= \sum \textstyle \frac{r!}{j_1! j_2! \cdots j_{r-l+1}!} \left( \frac{x_1}{1!} \right)^{j_1} \left( \frac{x_2}{2!} \right)^{j_2} \cdots \left( \frac{x_{r-l+1}}{(r-l+1)!} \right)^{j_{r-k+1}}
	\end{equation}
	are Bell polynomial, the sum is over all $(j_1,j_2,\ldots,j_{r-l+1}) \in \mathds{N}^{r-l+1}$, such that $j_1 + j_2 + \ldots + j_{r-l+1} = l$, és $j_1 + 2j_2 + \ldots + (r-l+1) j_{r-l+1} = r$.
\end{allitas}



\begin{proof}[Proof of statement \ref{121}]
	Let us suppose that $0 < x < a_N+1$, then $0 < y := \gamma(x) < \gamma(a_N+1)$. The inverse of the defined $\gamma$ function will always be differentiable, this is the case when $r=1$ ($0 < \beta_k < 1$). We specify the higher order derivatives with the help of the Fa\`{a} di Bruno formula: let $g(x) := \gamma^{-1}(\gamma(x)) = x$, then the derivatives of $g$, as a composition function are the following:
	\begin{equation*} 
		g^{(r)}(x) = \sum_{l=1}^{r} (\gamma^{-1})^{(l)}(\gamma(x)) \cdot B_{r,l}(\gamma'(x), \gamma''(x), \ldots, \gamma^{(r-l+1)}(x)),
	\end{equation*}
	where $B_{r,l}$ are the Bell polynomial from (\ref{564}).
	
	$g' = 1$, and all of the higher order derivatives of $g$ is equal to $0$, so we can take the following system of linear equations for $(\gamma^{-1})^{(l)}$-s:
	\begin{align*}
		(\gamma^{-1})'(\gamma(x)) \cdot \gamma'(x) & = 1 \\ 
		\sum_{l=1}^{n} (\gamma^{-1})^{(l)}(\gamma(x)) \cdot B_{n,l}(\gamma'(x), \gamma''(x), \ldots, \gamma^{(n-l+1)}(x)) & = 0, \ \ \ 2 \leq n \leq r 
	\end{align*}
	The coefficients of the diagonal terms are $B_{l,l}(\gamma'(x)) = (\gamma'(x))^l$, the determinant of the system is $D = \prod_{l=1}^{r} B_{l,l} = (\gamma'(x))^{\frac{r(r+1)}{2}}$. If $\gamma'(x) \neq 0, \infty$, then we can express $(\gamma^{-1})^{(r)}(y)$ from the equation system: with the Cramer-rule we get that if $r \geq 2$, then
	\begin{equation} \label{126}
		(\gamma^{-1})^{(r)}(y) = \frac{(-1)^{r+1}}{(\gamma'(x))^{\frac{r(r+1)}{2}}} \cdot \det \underbrace{\left[
			\begin{array}{ccccc}
				B_{2,1} & B_{2,2} & 0 & \cdots & 0 \\
				B_{3,1} & B_{3,2} & B_{3,3} & \ddots & \vdots \\
				\vdots & \vdots & \vdots & \ddots & 0 \\
				B_{r-1,1} & B_{r-1,2} & B_{r-1,3} & \cdots & B_{r-1,r-1} \\
				B_{r,1} & B_{r,2} & B_{r,3} & \cdots & B_{r,r-1}
			\end{array}
		\right]}_{\mathbf{B}},
	\end{equation}
	holds, where $B_{i,j} = B_{i,j}(\gamma'(x),\gamma''(x), \ldots, \gamma^{(i-j+1)}(x))$ if $i \geq j$. Based on this, if $y \neq \gamma(a_k)$, $k=1,\ldots,N$, then $\gamma^{-1}$ will be $r$-times differentiable, because $\gamma$ is $r$-times differentiable except to the points $a_1, \ldots, a_N$. However, taken the limit, we get that there is no problem in these points either. If $x \neq a_k$, $k=1,\ldots,N$, then the $j$th derivative of $\gamma$ is
	\begin{equation*} 
		\gamma^{(j)}(x) = \sum_{k=1}^{N} \beta_k (\beta_k -1) \cdots (\beta_k - (j-1)) \cdot |x-a_k|^{\beta_k - j} (\mathrm{sgn} (x-a_k))^{j+1}.
	\end{equation*}
	Let us look at the limit of (\ref{126}) for example in $y = \gamma(a_k)$. There are products of $B_{n,l}$-type expressions in the numerator, and it is sufficient to investigate the exponents of $|x - a_k|$-s, because the other $|x - a_j|$-s are not zero finite numbers, if $j \neq k$. In $\gamma^{(i)}$ the exponent of $|x-a_k|$ is $\beta_k - i$. This is because
	\begin{equation*} 
		B_{n,l}(\gamma',\gamma'', \ldots, \gamma^{(n-l+1)}) 
		= \sum \textstyle \frac{l!}{j_1! j_2! \cdots j_{n-l+1}!} \left( \frac{\gamma'}{1!} \right)^{j_1} \left( \frac{\gamma''}{2!} \right)^{j_2} \cdots \left( \frac{\gamma^{(n-l+1)}}{(n-l+1)!} \right)^{j_{n-l+1}},
	\end{equation*}
	where the sum is understood as above, in $(\gamma')^{j_1} (\gamma'')^{j_2} \ldots (\gamma^{(n-l+1)})^{j_{n-l+1}}$ the smallest exponent of $|x-a_k|$ is
	\begin{equation*} 
		(\beta_k - 1) j_1 + (\beta_k -2) j_2 + \ldots + (\beta_k - (n-l+1)) j_{n-l+1} = \beta_k l - n.
	\end{equation*}
	Based on the definition of the determinant
	\begin{equation*} 
		\det \mathbf{B} = \sum_{\sigma \in S_{r-1}} (-1)^{I(\sigma)} B_{2,\sigma(1)} B_{3, \sigma(2)} \cdots B_{r,\sigma(r-1)}
	\end{equation*}
	with the limitation that $B_{i,j} \equiv 0$ if $i < j$. If the $\sigma$ permutation is such that $\sigma(i) \leq i+1$, $i=1,\ldots,r-1$, then the smallest exponent of $|x-a_k|$ in $B_{2,\sigma(1)} B_{3, \sigma(2)} \cdots B_{r,\sigma(r-1)}$ is
	\begin{equation*} 
		(\beta_k \sigma(1) - 2) + (\beta_k \sigma(2) - 3) + \ldots + (\beta_k \sigma(r-1) - r) 
		= \beta_k \frac{r(r-1)}{2} - \frac{r(r+1)}{2} + 1,
	\end{equation*}
	in the other cases the term in the determinant is zero. In the denominator of the fraction giving $(\gamma^{-1})^{(r)}$ from the terms containing $|x-a_k|$ as factor, the smallest exponent of $|x-a_k|$ is $(\beta_k - 1)\frac{r(r+1)}{2}$. Let us bring in the factor $|x-a_k|^{(1-\beta_k) \frac{r(r+1)}{2}}$. Then in the denominator all of the exponents of $|x-a_k|$-s are nonnegative, and in the numerator the smallest exponent of $|x-a_k|$ is
	\begin{equation*} 
		\beta_k \frac{r(r-1)}{2} - \frac{r(r+1)}{2} + 1 + (1 - \beta_k) \frac{r(r+1)}{2} = - \beta_k r + 1 > 0,
	\end{equation*}	
	because of the assumption $\beta_k < \frac{1}{r}$ for all $1 \leq k \leq N$. Thus we get that the limit of $(\gamma^{-1})^{(r)}(y)$ in $\gamma(a_k)$ is zero, $k=1,\ldots,N$, so $\gamma^{-1}$ is $r$-times differentiable.
\end{proof}

\begin{proof}[Proof of statement \ref{197}]
On the left side in the inequality which we wanted to prove is the function of the tangent line of $\gamma$ in $0$, on the right side is the function of the secant from the $(0,0)$ to the $(a_1,\gamma(a_1))$ point. Because of $\gamma$ is convex on $[0,a_1]$, the statement is true.
\end{proof}



From here we prove the main results. To prove theorem \ref{50} we must have the following lemmas:

\begin{lemma} \label{37}
	Let $\alpha \geq 0$, 
	$A_1 \geq \frac{1}{4}$, $A_2 > 0$ are real numbers, $0 < h \leq \gamma \left( \sqrt[4]{\frac{A_2}{A_1}} \frac{1}{\sqrt{2r}} \right)$, $r \geq 1$ is an integer. If $x \in \left[4 A_1 r^2 (\gamma^{-1}(h))^2, \frac{A_2}{(\gamma^{-1}(h))^{2}} \right]$ and
		$y \in \left[ x - r \gamma^{-1}(h) \sqrt{x}, x + r \gamma^{-1}(h) \sqrt{x} \right]$,
	then there exists $c$ and $C$ (depending on $A_1$, $A_2$, $\alpha$ and $r$) positive constants such that 
		$w_{\alpha}(y) \sim w_{\alpha}(x).$ 
\end{lemma}

\begin{proof}
	By the assumptions $|x-y| \leq r \gamma^{-1}(h) \sqrt{x}$ and $\gamma^{-1}(h) \leq \frac{\sqrt{x}}{2r \sqrt{A_1}}$, so
	\begin{equation*} 
		x \leq y + |x-y| \leq y + r \gamma^{-1}(h) \sqrt{x} \leq y + \frac{x}{2 \sqrt{A_1}},
	\end{equation*}
	and
	\begin{equation*} 
		y \leq x + |x-y| \leq x + r \gamma^{-1}(h) \sqrt{x} \leq x + \frac{x}{2 \sqrt{A_1}},
	\end{equation*}
	that is,
	\begin{equation} \label{41}
		\frac{2 \sqrt{A_1} - 1}{2 \sqrt{A_1}} \leq \frac{y}{x} \leq \frac{2 \sqrt{A_1} + 1}{2 \sqrt{A_1}}.
	\end{equation}
	By the assumptions $\sqrt{x} \leq \frac{\sqrt{A_2}}{\gamma^{-1}(h)}$ too, thus
	\begin{equation} \label{453}
		e^{-r \sqrt{A_2}} \leq e^{x-y} \leq e^{r \sqrt{A_2}}.
	\end{equation}
	Since $\alpha \geq 0$, we get from (\ref{41}) and (\ref{453}) that the statement is true.
\end{proof}

\begin{lemma} \label{359}
	\cite{Ditzian1980} Let $0 \leq a \leq b$. If $f, f^{(r)} \in L^p(a,b)$, $1 \leq p \leq \infty$, then for all $0 \leq k \leq r$
	\begin{equation*} 
		\| f^{(k)} \|_{L^p(a,b)} \leq M(r,k) \left( \| f \|_{L^p(a,b)} (b-a)^{-k} + \| f^{(r)} \|_{L^p(a,b)} (b-a)^{r-k} \right)
	\end{equation*}
	holds, where $M$ is independent of $p$ and $[a,b]$.
\end{lemma}

After all this we can prove theorem \ref{50}.

\begin{proof}[Proof of theorem \ref{50}]
	From now $C$ is always a universal constant, which is independent of $f$ and $t$. We prove the theorem in several steps, we prove at first (\ref{51}).
	
Let us consider the
	\begin{equation*} 
		I_{rh,\gamma} = \left[4A_1r^2 (\gamma^{-1}(h))^2, \frac{A_2}{(\gamma^{-1}(h))^{2}} \right]
	\end{equation*}
	interval, where $\frac{1}{4} \leq A_1 \frac{2 \sqrt{A_1}}{1 + 2 \sqrt{A_1}}$, $0 < A_2$. Since $\frac{2 \sqrt{A_1}}{1 + 2 \sqrt{A_1}} < 1$, $\frac{1}{4} < A_1$ is also true. Let $0 < h \leq t \leq T$, where
	\begin{multline} \label{250}
	T < \min \left\{ \textstyle \frac{1}{2}, a_1, \gamma(1), \gamma \left( \sqrt[4]{\frac{A_2}{A_1}} \frac{1}{\sqrt{2r}} \right), \gamma \left( \frac{a_1 \sqrt{A_1}}{r(1 + 2 \sqrt{A_1})} \cdot \sqrt{\frac{2 \sqrt{A_1}}{1 + 2 \sqrt{A_1}}} \right), \right. \\
	\textstyle \left. \gamma \left( \sqrt{\frac{2 \sqrt{A_1}}{(1 + 2 \sqrt{A_1})(a_N + 1)}} \right), \gamma \left( \frac{\sqrt{a_N+1}}{4 A_1 r} \right), \gamma \left( \frac{2 \sqrt{A_1}}{1 + 2 \sqrt{A_1}} \sqrt{\frac{A_2}{a_N + 1}} \right) \right\}.
	\end{multline}
	From the 5th assumption we get $h < \gamma \left( \frac{a_1 \sqrt{A_1}}{r(1 + 2 \sqrt{A_1})} \right) \leq \gamma \left( \frac{a_1 \sqrt{A_1}}{1 + 2 \sqrt{A_1}} \right) < \gamma \left( \frac{a_1}{2} \right)$, because $\frac{\sqrt{A_1}}{1 + 2 \sqrt{A_1}} < \frac{1}{2}$ and $r \geq 1$. 
	To prove (\ref{51}), we define such a $G_{\gamma,h} \in W_{\gamma}^{r,p}(I_{rh,\gamma})$ function that
	\begin{equation*} 
			\sup_{0 < h \leq t} \left\{ \| (f-G_{\gamma,h}) w_{\alpha} \|_{L^p(I_{rh,\gamma})} + h^r \| (G_{\gamma,h})^{(r)}_{\gamma} \varphi^r w_{\alpha} \|_{L^p(I_{rh,\gamma})} \right\} \leq C \sum_{n=1}^{r} t^{r-n} \Omega_{\gamma,\varphi}^{n} (f,t)_{w_{\alpha},p},
	\end{equation*}
	because of the (\ref{13}) definition of $\widetilde{K}_{\gamma,r,\varphi}(f,t^r,L^p_{w_{\alpha}},W_{\gamma}^{r,p})$ the estimation (\ref{51}) follows.
	
	\begin{lepes} \label{573}
	To construct $G_{\gamma,h} \in W_{\gamma}^{r,p}(I_{rh,\gamma})$ let us define a system of points $t_0,t_1, \ldots, t_{j+1}$ such that
	\begin{equation*} 
		t_0 := 4A_1r^2 (\gamma^{-1}(h))^2 < t_1 < t_2 < \ldots t_j < \frac{A_2}{(\gamma^{-1}(h))^{2}} \leq t_{j+1} =: \frac{A_2A}{(\gamma^{-1}(h))^{2}},
	\end{equation*}
	and
	\begin{equation} \label{52}
		\frac{1}{2 r} \leq \frac{t_{i+1} - t_i}{\gamma^{-1}(h) \sqrt{t_i}} \leq r, \ \ \ i = 0, \ldots, j.
	\end{equation}
	With this we get a covering of $I_{rh,\gamma}$, $I_{rh,\gamma} \subseteq [t_0,t_{j+1}]$. The following is a simpley consequence of (\ref{52}): for all $i=0,\ldots,j$
	\begin{equation} \label{238}
		1 \leq \frac{t_{i+1}}{t_i} \leq r \gamma^{-1}(h) \frac{1}{\sqrt{t_i}} + 1 \leq r \gamma^{-1}(h) \frac{1}{2 \sqrt{A_1} r \gamma^{-1}(h)} + 1 = \frac{1 + 2 \sqrt{A_1}}{2 \sqrt{A_1}}
	\end{equation}
	holds, so $t_i \sim t_{i+1}$ for all $i$. We get from (\ref{238}), that
	\begin{equation*} 
		\frac{A_2}{(\gamma^{-1}(h))^2} \leq t_{j+1} \leq \frac{1 + 2 \sqrt{A_1}}{2 \sqrt{A_1}} t_{j} < \frac{1 + 2 \sqrt{A_1}}{2 \sqrt{A_1}} \cdot \frac{A_2}{(\gamma^{-1}(h))^2},
	\end{equation*}
	so $1 \leq A < \frac{1 + 2 \sqrt{A_1}}{2 \sqrt{A_1}}$. Since $\gamma^{-1}$ is strictly increasing and $h < \gamma \left( \frac{\sqrt{a_N+1}}{4 A_1 r} \right)$, we can see that $t_1 < a_N + 1$, and since $h < \gamma \left( \frac{2 \sqrt{A_1}}{1 + 2 \sqrt{A_1}} \sqrt{\frac{A_2}{a_N + 1}} \right)$, we can get $t_{j-1} > a_N + 1$.
	According to these we have $t_0 < t_1 < a_N+1 < t_{j-1}$. 
	Let $M$ denote the index $1 \leq M \leq j-2$, for which $t_M < a_N+1 \leq t_{M+1}$ is satisfied.
	
	Later on we are working on $[t_0,t_{j+1}]$. Since $h <  \gamma \left( \sqrt[4]{\frac{A_2}{A_1}} \frac{1}{\sqrt{2r}} \right)$, because of the strictly increasing property of $\gamma$, $h <  \gamma \left( \sqrt[4]{\frac{A_2A}{A_1}} \frac{1}{\sqrt{2r}} \right)$ is also true, so $[t_0,t_{j+1}]$ is a $I_{rh,\gamma}$-type interval too, for which we can apply lemma \ref{37}. If $x,y \in [t_i,t_{i+1}]$, $i=0,\ldots,j$, then
	\begin{equation*} 
		|x-y| \leq t_{i+1} - t_i \leq r \gamma^{-1}(h) \sqrt{t_i} \leq r \gamma^{-1}(h) \sqrt{x},
	\end{equation*}
	on the other hand $x,y \in  \left[4A_1r^2 (\gamma^{-1}(h))^2, \frac{A_2A}{(\gamma^{-1}(h))^{2}} \right]$, so by lemma \ref{37} $w_{\alpha}(x) \sim w_{\alpha}(y)$.
	Let us define the following function:
	\begin{equation} \label{56}
		\psi(x) := \left\{
			\begin{array}{rl}
				0, & \text{if } x \leq 0, \\
				1, & \text{if } 1 \leq x,
			\end{array}
		\right.
	\end{equation}
	and $\psi \in C^{\infty}(\mathds{R})$, monotone increasing. Let $y_k := \gamma^{-1} \left( \frac{\gamma(t_k) + \gamma(t_{k+1})}{2} \right)$, $k = 0, \ldots, j$. Then $t_{k} < y_k < t_{k+1}$ and
	\begin{equation} \label{260}
		\gamma(t_{k+1}) - \gamma(y_k) 
		= \frac{1}{2} (\gamma(t_{k+1}) - \gamma(t_k)).
	\end{equation}
	Furthermore let
	\begin{equation*} 
		\psi_k (x) = \psi \left( \frac{\gamma(x)-\gamma(y_k)}{\gamma(t_{k+1}) - \gamma(y_k)} \right) = \left\{
			\begin{array}{rl}
				0, & \text{if } x \leq y_k, \\
				1, & \text{if } t_{k+1} \leq x,
			\end{array}
		\right.
	\end{equation*}
	monotone increasing, $k=1,\ldots,j-1$. For $k=0$ and $k=j$ let $\psi_0(x) \equiv 1$, $\psi_j(x) \equiv 0$.
	
	Let us introduce the following $\gamma$-Steklov function:
	\begin{equation*} 
		f_{\gamma,\tau,\varphi(t)}(x) 
		:= r^r \int_{0}^{\frac{1}{r}} \cdots \int_{0}^{\frac{1}{r}} \left( \sum_{l=1}^{r} \textstyle (-1)^{l+1} \binom{r}{l} f(x + l \gamma^{-1}(\tau) \sqrt{t}(u_1 + \ldots + u_r)) \displaystyle \right) \, \mathrm{d}u_1 \ldots \mathrm{d}u_r,
	\end{equation*}
	then with the help of this let
	\begin{equation*} 
		F_{\gamma,h,k}(x) := \frac{2a}{h} \int_{\frac{h}{2a}}^{\frac{h}{a}} f_{\gamma,\tau,\varphi(t_{k-1})}(x) \, \mathrm{d} \tau, \ \ \ k = 1,\ldots,j,
	\end{equation*}
	where
	\begin{equation} \label{281}
		a > \max \left\{ 1, \frac{2 r^2 \gamma(a_1)}{a_1 \sum_{k=1}^{N} \beta_k a_k^{\beta_k-1}} 
		\right\}.
	\end{equation}
	(This choice of $a$ will be useful in the further estimates). 
	Finally let
	\begin{equation} \label{61}
		G_{\gamma,h}(x) := \sum_{k=1}^{j} F_{\gamma,h,k}(x) \psi_{k-1}(x)(1-\psi_{k}(x)).
	\end{equation}
		From the definitions of $\psi_k$-s it is easy to see that if $i=1,\ldots,j-1$, $x \in [t_i,t_{i+1}]$, then
	\begin{equation} \label{64}
		G_{\gamma,h}(x) = F_{\gamma,h,i}(x)(1 - \psi_i(x)) + F_{\gamma,h,i+1}(x) \psi_i(x),
	\end{equation}
	and
	\begin{equation} \label{239}
		G_{\gamma, h}(x) = F_{\gamma,h,1}(x), \text{ if } x \in [t_0,t_1] \text{ and } G_{\gamma,h}(x) = F_{\gamma,h,j}(x), \text{ if } x \in [t_j,t_{j+1}].
	\end{equation}
	It is equivalent with (\ref{64}), that if $x \in [t_i,t_{i+1}]$, $i=1,\ldots,j-1$, then
	\begin{equation} \label{216}
		G_{\gamma,h}(x) = F_{\gamma,h,i}(x) + \psi_i(x) (F_{\gamma,h,i+1}(x) - F_{\gamma,h,i}(x)).
	\end{equation}
	We show that $G_{\gamma,h} \in W_{\gamma}^{r,p}(I_{rh, \gamma})$. By (\ref{239}) we get that
	\begin{equation*} 
		(G_{\gamma, h})_{\gamma}^{(r)}(x) = (F_{\gamma,h,1})^{(r)}_{\gamma}(x), \text{ if } x \in [t_0,t_1]
	\end{equation*}
	and
	\begin{equation*} 
		(G_{\gamma,h})_{\gamma}^{(r)}(x) = (F_{\gamma,h,j})_{\gamma}^{(r)}(x), \text{ if } x \in [t_j,t_{j+1}],
	\end{equation*}
	and if $x \in [t_i,t_{i+1}]$, $i=1,\ldots,j-1$, then by (\ref{216}) and statement \ref{140}
	\begin{equation} \label{217}
		(G_{\gamma,h})^{(r)}_{\gamma}(x) 
		= (F_{\gamma,h,i})^{(r)}_{\gamma}(x) + \sum_{k=0}^{r} \binom{r}{k} ((F_{\gamma,h,i+1})^{(k)}_{\gamma} - (F_{\gamma,h,i})^{(k)}_{\gamma})(x) (\psi_i)^{(r-k)}_{\gamma}(x).
	\end{equation}
	By statement \ref{25} we have
	\begin{equation*} 
		(\psi_i)^{(r-k)}_{\gamma}(x) = \frac{1}{(\gamma(t_{i+1}) - \gamma(y_i))^{r-k}} \psi^{(r-k)} \left( \frac{\gamma(x) - \gamma(y_i)}{\gamma(t_{i+1}) - \gamma(y_i)} \right)
	\end{equation*}
	is a continuous function, and if $(F_{\gamma,h,i})^{(k)}$ and $(\gamma^{-1})^{(k)}(\gamma(x))$ are well defined, then by the statement \ref{17} and by the Fa\`{a} di Bruno formula for all $k=1,\ldots,r$ we get
	\begin{equation*} 
		(F_{\gamma,h,i})^{(k)}_{\gamma}(x) 
		= (F_{\gamma,h,i} \circ \gamma^{-1})^{(k)}(\gamma(x)) = \sum_{l=1}^{k} (F_{\gamma,h,i})^{(l)}(x) B_{k,l}(x),
	\end{equation*}
	where
	\begin{equation*} 
		B_{k,l}(x) = B_{k,l}((\gamma^{-1})'(\gamma(x)), (\gamma^{-1})''(\gamma(x)), \ldots, (\gamma^{-1})^{(k-l+1)}(\gamma(x))) 
	\end{equation*}
	is the corresponding Bell-polynomial, which is a continous function on $[t_0,t_{j+1}] \cap (0,a_N + 1]$ and on $[t_0,t_{j+1}] \cap [a_N + 1,\infty)$ by statement \ref{121}. So from the properties of the Steklov function and the remarks above we get from (\ref{217}) that
	\begin{equation} \label{222d}
		\infty > 
		\int_{t_0}^{t_{j+1}} |(G_{\gamma,h})^{(r)}_{\gamma}(x)|^p \, \mathrm{d} x.
	\end{equation}
	However, if (\ref{222d}) is true, then $G_{\gamma,h} \in W_{\gamma}^{r,p}(I_{rh, \gamma})$ (even $G_{\gamma,h} \in W_{\gamma}^{r,p}([t_0,t_{j+1}])$) holds.
	
	\end{lepes}
	
	\begin{lepes} \label{574}
	We will show that if $0 < h \leq t \leq T$, $1 \leq p \leq \infty$, then
	\begin{equation} \label{62}
		\| (f - G_{\gamma,h}) w_{\alpha} \|_{L^p(I_{rh, \gamma})} \leq C \cdot \Omega_{\gamma,\varphi}^{r}(f,t)_{w_{\alpha},p}
	\end{equation}
	holds, where this last expression we define with the help of interval
	\begin{equation*} 
		I_{rh,\gamma}' = \left[4A_1r^2 (\gamma^{-1}(h))^2, \frac{\frac{1 + 2 \sqrt{A_1}}{2 \sqrt{A_1}} A_2}{(\gamma^{-1}(h))^{2}} \right].
	\end{equation*}
	At first let $1 \leq p < \infty$, then
	\begin{equation} \label{63}
		\| (f - G_{\gamma,h}) w_{\alpha} \|_{L^p(I_{rh, \gamma})}^{p} 
		\leq \sum_{i=0}^{j} \int_{t_i}^{t_{i+1}} |f(x) - G_{\gamma,h}(x)|^p w_{\alpha}^p(x) \, \mathrm{d} x.
	\end{equation}
		
	So we get the following estimates: if $x \in [t_i,t_{i+1}]$, $i=1,\ldots,j-1$, then
	\begin{multline} \label{65}
		|f(x) - G_{\gamma,h}(x)| = |(f(x) - F_{\gamma,h,i}(x))(1-\psi_i(x)) + (f(x)-F_{\gamma,h,i+1}(x)) \psi_i(x)| \\
		\leq |f(x) - F_{\gamma,h,i}(x)| + |f(x) - F_{\gamma,h,i+1}(x)|,
	\end{multline}
	if $x \in [t_0,t_1]$, then
	\begin{equation} \label{66}
		|f(x) - G_{\gamma,h}(x)| = |f(x) - F_{\gamma,h,1}(x)|,
	\end{equation}
	finally if $x \in [t_j,t_{j+1}]$, then
	\begin{equation} \label{67}
		|f(x) - G_{\gamma,h}(x)| = |f(x) - F_{\gamma,h,j}(x)|.
	\end{equation}
	By (\ref{65}), (\ref{66}), (\ref{67}) and the Jensen inequality we get from (\ref{63}), that
	\begin{equation} \label{69}
		\| (f - G_{\gamma,h}) w_{\alpha} \|_{L^p(I_{rh, \gamma})}^{p} \leq 2^{p-1} \sum_{i=1}^{j} \int_{t_{i-1}}^{t_{i+1}} |f(x) - F_{\gamma,h,i}(x)|^p w_{\alpha}^p \, \mathrm{d} x.
	\end{equation}
	
	Therefore it is enough to estimate the
	\begin{equation} \label{70}
		I_i := \int_{t_{i-1}}^{t_{i+1}} |f(x) - F_{\gamma,h,i}(x)|^p w_{\alpha}^p(x) \, \mathrm{d} x
	\end{equation}
	integrals, where $i=1,\ldots,j$. Let us estimate the integrand: if $t_{i-1} \leq x \leq t_{i+1}$, then
	\begin{equation} \label{71}
		|f(x) - F_{\gamma,h,i}(x)| 
		\leq \frac{2a}{h} \int_{\frac{h}{2a}}^{\frac{h}{a}} r^r \int_{0}^{\frac{1}{r}} \cdots \int_{0}^{\frac{1}{r}} \left| \overrightarrow{\Delta}_{\gamma^{-1}(\tau) (u_1 + \ldots + u_r) \varphi(t_{i-1})}^{r} f(x) \right| \, \mathrm{d}u_1 \ldots \mathrm{d}u_r \mathrm{d} \tau.
	\end{equation}
	Let $u := u_1 + \ldots + u_r$, then $0 < u < 1$. Since
	\begin{multline*} 
		\left\{ (u_1,\ldots,u_r) \in \mathds{R}^r \mid 0 < u_i < \textstyle \frac{1}{r} \displaystyle , i=1,\ldots,r \right\} \\
		\subset \{ (u_1,\ldots,u_r) \in \mathds{R}^r \mid 0 < u_1 + \ldots + u_r < 1; \ 0 < u_i, i=1,\ldots,r \} =: H,
	\end{multline*}
	and $B_r := \lambda(H) < \infty$,	
	with the notation $B(r) := r^r B_r$ we can continue the estimation of (\ref{71}):
	\begin{equation} \label{72}
		|f(x) - F_{\gamma,h,i}(x)| \leq \frac{2a}{h} B(r) \int_{\frac{h}{2a}}^{\frac{h}{a}} \int_{0}^{1} \left| \overrightarrow{\Delta}_{\gamma^{-1}(\tau) u \varphi(t_{i-1})}^{r} f(x) \right| \, \mathrm{d}u \mathrm{d}\tau.
	\end{equation}
	After applying the Hölder inequality twice, by (\ref{72}):
	\begin{equation*} 
		|f(x) - F_{\gamma,h,i}(x)|^p \leq \frac{2a}{h} B^p(r) \int_{\frac{h}{2a}}^{\frac{h}{a}} \int_{0}^{1} \left| \overrightarrow{\Delta}_{\gamma^{-1}(\tau) u \varphi(t_{i-1})}^{r} f(x) \right|^p \, \mathrm{d}u \mathrm{d}\tau,
	\end{equation*}
	so
	\begin{equation} \label{74}
		I_i \leq B^p(r) \int_{t_{i-1}}^{t_{i+1}} \frac{2a}{h} \int_{\frac{h}{2a}}^{\frac{h}{a}} \int_{0}^{1} \left| \overrightarrow{\Delta}_{\gamma^{-1}(\tau) u \varphi(t_{i-1})}^{r} f(x) \right|^p w_{\alpha}^p(x) \, \mathrm{d}u \mathrm{d}\tau \, \mathrm{d} x.
	\end{equation}
	Let $\gamma(\gamma^{-1}(\tau) u) =: v$, then
	\begin{equation*} 
		\textstyle \gamma \left( \gamma^{-1} \left( \frac{h}{2a} \right) u  \right) \leq v \leq \gamma \left(\gamma^{-1} \left( \frac{h}{a} \right) u \right),
	\end{equation*}
	beacuse $\gamma$, $\gamma^{-1}$ are monotone increasing functions. Since $0 < u < 1$ we get that $0 < v < \frac{h}{a}$, and if
	\begin{equation*} 
		H := \left\{ (u,\tau) \in \mathds{R}^2 \mid 0 < u < 1, \ \textstyle \frac{h}{2a} \leq \tau \leq \frac{h}{a} \displaystyle \right\},
	\end{equation*}
	then $\lambda(H) = \frac{h}{2a} < h < 1$, so by (\ref{74}):
	\begin{multline} \label{75}
		I_i 
		\leq \frac{2a}{h} B^p(r) \lambda(H) \int_{t_{i-1}}^{t_{i+1}} \int_{0}^{\frac{h}{a}} \left| \overrightarrow{\Delta}_{\gamma^{-1}(v) \varphi(t_{i-1})}^{r} f(x) \right|^p w_{\alpha}^p(x) \, \mathrm{d}v \mathrm{d} x \\
		\leq \frac{2a}{h} B^p(r) \int_{t_{i-1}}^{t_{i+1}} \int_{0}^{\frac{h}{a}} \left| \overrightarrow{\Delta}_{\gamma^{-1}(v) \varphi(t_{i-1})}^{r} f(x) \right|^p w_{\alpha}^p(x) \, \mathrm{d}v \mathrm{d} x.
	\end{multline}
	Since $t_{i-1} \leq x \leq t_{i+1}$, then by (\ref{238})
	\begin{equation} \label{77}
		1 \leq \frac{x}{t_{i-1}} \leq \frac{t_{i+1}}{t_{i-1}} = \frac{t_{i+1}}{t_i} \cdot \frac{t_{i}}{t_{i-1}} \leq \left( \frac{1 + 2 \sqrt{A_1}}{2 \sqrt{A_1}} \right)^2,
	\end{equation}
	so $x \sim t_{i-1}$. In the inner integral let us change the variables: $\gamma^{-1}(v) \sqrt{t_{i-1}} = \gamma^{-1}(\nu) \sqrt{x}$:
	\begin{multline} \label{76}
		\int_{0}^{\frac{h}{a}} \left| \overrightarrow{\Delta}_{\gamma^{-1}(v) \varphi(t_{i-1})}^{r} f(x) \right|^p w_{\alpha}^p(x) \, \mathrm{d}v \\
		= \sqrt{\frac{x}{t_{i-1}}} \int_{0}^{\gamma \left( \gamma^{-1} \left( \frac{h}{a} \right) \sqrt{\frac{t_{i-1}}{x}} \right)} \left| \overrightarrow{\Delta}_{\gamma,\varphi,\nu}^{r} f(x) \right|^p w_{\alpha}^p(x) \gamma' \left( \gamma^{-1} (\nu) \sqrt{\frac{x}{t_{i-1}}} \right) \cdot (\gamma^{-1})'(\nu) \, \mathrm{d} \nu.
	\end{multline}
	We can do this, because $\gamma'$ and $(\gamma^{-1})'$ is bounded on $\left[0,\frac{a_1}{2}\right]$, and the domain of the integration is a subset of this interval.
	So there is a constant $C$ such that by (\ref{76})
	\begin{equation*} 
		\int_{0}^{\frac{h}{a}} \left| \overrightarrow{\Delta}_{\gamma^{-1}(v) \varphi(t_{i-1})}^{r} f(x) \right|^p w_{\alpha}^p(x) \, \mathrm{d}v \leq C \int_{0}^{\frac{h}{a}} \left| \overrightarrow{\Delta}_{\gamma,\varphi,\nu}^{r} f(x) \right|^p w_{\alpha}^p(x) \, \mathrm{d} \nu,
	\end{equation*}
	then by (\ref{75}):
	\begin{equation} \label{79}
		I_i \leq C^p \frac{a}{h} \int_{0}^{\frac{h}{a}}  \int_{t_{i-1}}^{t_{i+1}} \left| \overrightarrow{\Delta}_{\gamma,\varphi,\nu}^{r} f(x) \right|^p w_{\alpha}^p(x) \, \mathrm{d} x \mathrm{d} \nu,
	\end{equation}
	we used the Fubini theorem. So we get by (\ref{69}) and (\ref{79}), that
	\begin{multline*} 
		\| (f - G_{\gamma,h}) w_{\alpha} \|_{L^p(I_{rh, \gamma},\lambda_{\gamma})}^{p} 
		\leq C^p \frac{a}{h} \int_{0}^{\frac{h}{a}} \int_{4A_1r^2 (\gamma^{-1}(h))^2}^{\frac{A_2A}{(\gamma^{-1}(h))^2}}  \left| \overrightarrow{\Delta}_{\gamma,\varphi,\nu}^{r} f(x) \right|^p w_{\alpha}^p(x) \, \mathrm{d} x \, \mathrm{d} \nu \\
		\leq C^p \frac{a}{h} \int_{0}^{\frac{h}{a}} \left\| w_{\alpha} \overrightarrow{\Delta}_{\gamma,\varphi,\nu}^{r} f \right\|_{L^p \left( I_{rh,\gamma}' \right)}^p \, \mathrm{d} \nu \leq 
	C^p \Omega_{\gamma,\varphi}^{r}(f,t)_{w_\alpha,p}^{p},
	\end{multline*}
	so (\ref{62}) is true. Since the constant $C$ is independent of $p$, the statement is proved for $p=\infty$ too.
\end{lepes}

\begin{lepes} \label{575}
	Let $1 \leq p < \infty$ again. We show that if $0 < h \leq t \leq T$, then
	\begin{equation} \label{144}
		h^r \| (G_{\gamma,h})^{(r)}_{\gamma} \varphi^r w_{\alpha} \|_{L^p(I_{rh, \gamma})} \leq C \sum_{n=1}^{r} t^{r-n} \Omega_{\gamma,\varphi}^{n} (f,t)_{w_{\alpha},p},
	\end{equation}
	holds, where this latter expressions we define with the help of $I_{rh,\gamma}'$ again. Since
	\begin{equation} \label{147}
		h^{rp} \| (G_{\gamma,h})^{(r)}_{\gamma} \varphi^r w_{\alpha} \|_{L^p(I_{rh, \gamma})}^p \leq h^{rp} \sum_{i=0}^{j} \| (G_{\gamma,h})^{(r)}_{\gamma} \varphi^r w_{\alpha} \|_{L^p([t_i,t_{i+1}])}^{p},
	\end{equation}
	it is sufficient to estimate the integral only on the $[t_i,t_{i+1}]$, $i=0,\ldots,j$ intervals. By (\ref{239}) on the first and on the last interval of this type:
	\begin{equation} \label{164}
		h^{rp} \| (G_{\gamma,h})^{(r)}_{\gamma} \varphi^r w_{\alpha} \|_{L^p([t_0,t_{1}])}^{p} = h^{rp} \| (F_{\gamma,h,1})^{(r)}_{\gamma} \varphi^r w_{\alpha} \|_{L^p([t_0,t_{1}])}^{p},
	\end{equation}
	and
	\begin{equation} \label{165}
	h^{rp} \| (G_{\gamma,h})^{(r)}_{\gamma} \varphi^r w_{\alpha} \|_{L^p([t_j,t_{j+1}])}^{p} = h^{rp} \| (F_{\gamma,h,j})^{(r)}_{\gamma} \varphi^r w_{\alpha} \|_{L^p([t_j,t_{j+1}])}^{p}.
	\end{equation}
	If $x \in [t_i,t_{i+1}]$, $i=1,\ldots,j-1$, then by statement \ref{140} and by (\ref{216}), with the help of the Jensen inequality we get
	\begin{multline} \label{148}
		h^{rp} \| (G_{\gamma,h})^{(r)}_{\gamma} \varphi^r w_{\alpha} \|_{L^p([t_i,t_{i+1}])}^{p} \leq h^{rp} (r+2)^{p-1} \| (F_{\gamma,h,i})^{(r)}_{\gamma} \varphi^{r} w_{\alpha} \|_{L^p([t_i,t_{i+1}])}^{p} \\
		+ (r+2)^{p-1} \sum_{k=0}^{r} \binom{r}{k} \underbrace{h^{rp} \| (F_{\gamma,h,i+1} - F_{\gamma,h,i})^{(k)}_{\gamma} (\psi_i)^{(r-k)}_{\gamma} \varphi^{r} w_{\alpha} \|_{L^p([t_i,t_{i+1}])}^{p}}_{D_{k,i}}.
	\end{multline}
	Let us estimate in this latter case $D_{k,i}$-s. If $x \in [t_i,t_{i+1}]$, then by lemma \ref{37} $w_{\alpha}(x) \sim w_{\alpha}(t_i)$, on the other hand by (\ref{238}) 
	$t_i \sim t_{i+1}$, and $\varphi^r(x) \sim \varphi^r(t_i)$ are also true. Thus
	\begin{equation} \label{167}
		w_{\alpha}(t_i) \varphi^{r}(t_i) \sim w_{\alpha}(x) \varphi^{r}(x).
	\end{equation}
	By statement \ref{25} taking into account (\ref{260}) we get
	\begin{equation} \label{264}
		|(\psi_i)^{(r-k)}_{\gamma}(x)| \leq \underbrace{\sup_{0 \leq \nu \leq r} \| \psi^{(\nu)} \|_{\infty}}_{=: S} \frac{2^{r-k}}{(\gamma(t_{i+1}) - \gamma(t_i))^{r-k}},
	\end{equation}
	so we get the following as the estimation of $D_{k,i}$-s by (\ref{167}): 
	\begin{equation} \label{265}
		D_{k,i} 
		\leq C^p \frac{h^{rp} w_{\alpha}^{p}(t_i) \varphi^{rp}(t_i)}{(\gamma(t_{i+1}) - \gamma(t_i))^{(r-k)p}} \| (F_{\gamma,h,i+1} - F_{\gamma,h,i})^{(k)}_{\gamma} \|_{L^p([t_i,t_{i+1}])}^{p}.
	\end{equation}
	Let us suppose first, that $1 \leq i \leq M + 1$, then by (\ref{238}) we have $t_i \sim t_{i+1}$, so there is a constant $\mathcal{C}$ such that for all $1 \leq i \leq M + 1$
	$[t_{i},t_{i+1}] \subset \left[ 0, \mathcal{C} a_N \right] =: I$.
	If $1 \leq k \leq r$, then by statement \ref{17} and the Fa\`{a} di Bruno formula
	\begin{equation*} 
		(F_{\gamma,h,i})^{(k)}_{\gamma}(x) 
		= (F_{\gamma,h,i} \circ \gamma^{-1})^{(k)}(\gamma(x)) = \sum_{m=1}^{k} (F_{\gamma,h,i})^{(m)}(x) B_{k,m}(x),
	\end{equation*}	
	where $B_{k,m}(x) := B_{k,m}((\gamma^{-1})'(\gamma(x)), (\gamma^{-1})''(\gamma(x)), \ldots, (\gamma^{-1})^{(k-m+1)}(\gamma(x)))$.
	
	Since $[t_{i},t_{i+1}] \subset I$ and $1 \leq k \leq r$:
	\begin{equation*} 
		|B_{k,m}(x)| \leq \sup_{1 \leq m \leq r} \| B_{r,m} \|_{L^{\infty}(I)} \leq C.
	\end{equation*}
	Using this, by the Jensen inequality again we get
	\begin{equation} \label{362}
		\| (F_{\gamma,h,i+1} - F_{\gamma,h,i})^{(k)}_{\gamma} \|_{L^p([t_i,t_{i+1}])}^{p} 
		\leq C^p \sum_{m=1}^{k} \left\| (F_{\gamma,h,i+1} - F_{\gamma,h,i})^{(m)} \right\|_{L^p([t_i,t_{i+1}])}^p.
	\end{equation}
	Using lemma \ref{359} and the Jensen inequality, we get from (\ref{362}):
	\begin{multline} \label{364}
		\| (F_{\gamma,h,i+1} - F_{\gamma,h,i})^{(k)}_{\gamma} \|_{L^p([t_i,t_{i+1}])}^{p} \\
		\leq C^p \sum_{m=1}^{k} (t_{i+1} - t_i)^{-pm} \left\| F_{\gamma,h,i+1} - F_{\gamma,h,i} \right\|_{L^p([t_i,t_{i+1}])}^p \\
		+ C^p (t_{i+1} - t_i)^{pr} \sum_{m=1}^{k} (t_{i+1} - t_i)^{-pm} \| (F_{\gamma,h,i+1} - F_{\gamma,h,i})^{(r)} \|_{L^p([t_i,t_{i+1}])}^p.
	\end{multline}
	By (\ref{52})
	\begin{equation*} 
		(t_{i+1} - t_i)^{-p} 
		\leq \left( \frac{2r}{\gamma^{-1}(h) \sqrt{t_i}} \right)^p,
	\end{equation*}
	furthermore by (\ref{238}), since $h \leq \gamma \left( \sqrt{\frac{2 \sqrt{A_1}}{(1 + 2 \sqrt{A_1})(a_N + 1)}} \right)$:
	\begin{equation*} 
		t_i \leq \frac{t_{M+1}}{t_M} \cdot t_M \leq \frac{1 + 2 \sqrt{A_1}}{2 \sqrt{A_1}} (a_N + 1) \leq \frac{1}{(\gamma^{-1}(h))^2},
	\end{equation*}
	so $\sqrt{t_i} \leq \frac{1}{\gamma^{-1}(h)}$, thus
	\begin{equation*} 
		\frac{2r}{\gamma^{-1}(h) \sqrt{t_i}} \geq \frac{2r}{\gamma^{-1}(h) \frac{1}{\gamma^{-1}(h)}} = 2r > 1.
	\end{equation*}
	Since the exponential function is strictly increasing, if its base is bigger than $1$:
	\begin{equation*} 
		 \sum_{m=1}^{k} (t_{i+1} - t_i)^{-pm} 
		 \leq \sum_{m=1}^{k} \left( \frac{2r}{\gamma^{-1}(h) \sqrt{t_i}} \right)^{pk} = \frac{k (2r)^{pk}}{(\gamma^{-1}(h) \sqrt{t_i})^{pk}}.
	\end{equation*}
	Let us write this to (\ref{364}), then by (\ref{265}) we get for the estimation of $D_{k,i}$-s:
	\begin{multline} \label{369}
		D_{k,i} \leq C^p \frac{h^{rp} w_{\alpha}^p(t_i) \varphi^{rp}(t_i)}{(\gamma(t_{i+1}) - \gamma(t_i))^{(r-k)p}} \cdot \frac{1}{(\gamma^{-1}(h) \sqrt{t_i})^{pk}} \left\| F_{\gamma,h,i+1} - F_{\gamma,h,i} \right\|_{L^p([t_i,t_{i+1}])}^p \\
		+ C^p \frac{h^{rp} w_{\alpha}^p(t_i) \varphi^{rp}(t_i)}{(\gamma(t_{i+1}) - \gamma(t_i))^{(r-k)p}} \cdot \frac{(t_{i+1}-t_i)^{rp}}{(\gamma^{-1}(h) \sqrt{t_i})^{pk}} \| (F_{\gamma,h,i+1} - F_{\gamma,h,i})^{(r)} \|_{L^p([t_i,t_{i+1}])}^p.
	\end{multline}
	Since $h < \gamma(a_1)$, then by (\ref{200}):
	\begin{equation} \label{267}
		h^{rp} \leq \left( \frac{\gamma(a_1)}{a_1} \right)^{rp} (\gamma^{-1}(h))^{rp}.
	\end{equation}
	Using (\ref{267}) and the Lipschitz property of $\gamma^{-1}$ on the interval $I$:
	\begin{multline} \label{370}
		\frac{h^{rp} w_{\alpha}^p(t_i) \varphi^{rp}(t_i)}{(\gamma(t_{i+1}) - \gamma(t_i))^{(r-k)p}} \cdot \frac{1}{(\gamma^{-1}(h) \sqrt{t_i})^{pk}} \\
		\leq \left( \frac{\gamma(a_1)}{a_1} \right)^{rp} \underbrace{\left( \frac{\gamma^{-1}(h) \sqrt{t_i}}{t_{i+1} - t_i} \right)^{(r-k)p}}_{\leq (2r)^{(r-k)p}} \underbrace{\left( \frac{t_{i+1} - t_i}{\gamma(t_{i+1}) - \gamma(t_i)} \right)^{(r-k)p}}_{\text{bounded}} w_{\alpha}^{p}(t_i),
	\end{multline}
	furthermore by (\ref{52}), as $h < \gamma(a_1)$:
	\begin{equation} \label{371}
		(t_{i+1} - t_i)^{rp} \leq r^{rp} (\gamma^{-1}(h) \sqrt{t_i})^{rp} \leq C^p h^{rp} \varphi^{rp}(t_i).
	\end{equation}
	By (\ref{370}) and (\ref{371}), as $w_{\alpha}(t_i) \sim w_{\alpha}(x)$, using (\ref{167}) we get from (\ref{369}), that if $k=1,\ldots,r$, $i=1,\ldots,M$, then
	\begin{multline} \label{372}
		D_{k,i} \leq C^p \left\| (F_{\gamma,h,i+1} - F_{\gamma,h,i}) w_{\alpha} \right\|_{L^p([t_i,t_{i+1}])}^p \\
		+ C^p h^{rp} \| (F_{\gamma,h,i+1} - F_{\gamma,h,i})^{(r)} w_{\alpha} \varphi^r \|_{L^p([t_i,t_{i+1}])}^p.
	\end{multline}
	If $k=0$, then using the estimation (\ref{370}) in the case $k=0$, since $w_{\alpha}(t_i) \sim w_{\alpha}(x)$, we get the following:
	\begin{multline} \label{268}
		D_{0,i} \leq C^p \frac{h^{rp} w_{\alpha}^{p}(t_i) \varphi^{rp}(t_i)}{(\gamma(t_{i+1}) - \gamma(t_i))^{rp}} \| F_{\gamma,h,i+1} - F_{\gamma,h,i} \|_{L^p([t_i,t_{i+1}])}^{p} \\
		\leq C^p \| (F_{\gamma,h,i+1} - F_{\gamma,h,i}) w_{\alpha} \|_{L^p([t_i,t_{i+1}])}^{p}.
	\end{multline}
	
	We get similar integrals, if $M+2 \leq i \leq j-1$. Namely in this case $t_i \geq t_{M+2} > t_{M+1} > a_N+1$, so $[t_i,t_{i+1}] \subset [a_N+1,\infty)$, and in this interval $\gamma(x) = \mathcal{C}_1 x + \mathcal{C}_2$, where the constants $\mathcal{C}_1$ and $\mathcal{C}_2$ are defined by (\ref{204}) and (\ref{205}). By statement \ref{337} the (\ref{265}) estimation is now the following:
	\begin{equation*} 
		D_{k,i} \leq C^p \frac{h^{rp} w_{\alpha}^p(t_i) \varphi^{rp}(t_i)}{\mathcal{C}_1^{rp} (t_{i+1} - t_i)^{(r-k)p}} \| (F_{\gamma,h,i+1} - F_{\gamma,h,i})^{(k)} \|^p_{L^p([t_i,t_{i+1}])}.
	\end{equation*}
	Then by using lemma \ref{359}, the Jensen inequality, estimation (\ref{267}) and the equivalences (\ref{167}) and $w_{\alpha}(t_i) \sim w_{\alpha}(x)$, we get with a similar (but more simple) procedure as in (\ref{369}) and (\ref{370}), that	
	\begin{multline} \label{378}
		D_{k,i} \leq C^p \| (F_{\gamma,h,i+1} - F_{\gamma,h,i}) w_{\alpha} \|^p_{L^p([t_i,t_{i+1}])} \\
		+ C^p h^{rp} \| (F_{\gamma,h,i+1} - F_{\gamma,h,i})^{(r)} w_{\alpha} \varphi^r \|^p_{L^p([t_i,t_{i+1}])}
	\end{multline}
	is also true if $k=0,\ldots,r$, $i=M+1,\ldots,j-1$. So we get with (\ref{268}) and (\ref{372}), that (\ref{378}) holds if $k=0,\ldots,r$, $i=1,\ldots,j-1$.
	With the triangular inequality and notation (\ref{70}), integrating on a grater domain we got for the first term of the sum in (\ref{378}) by (\ref{79}):
	\begin{multline} \label{242}
		C^p \| (F_{\gamma,h,i+1} - F_{\gamma,h,i}) w_{\alpha} \|^p_{L^p([t_i,t_{i+1}])} \leq C^p (I_{i+1}+I_i) \\
		\leq C^p \frac{a}{h} \int_{0}^{\frac{h}{a}} \left( \int_{t_{i}}^{t_{i+2}} \left| \overrightarrow{\Delta}_{\gamma,\varphi,\nu}^{r} f(x) \right|^p w_{\alpha}^p(x) \, \mathrm{d} x + \int_{t_{i-1}}^{t_{i+1}} \left| \overrightarrow{\Delta}_{\gamma,\varphi,\nu}^{r} f(x) \right|^p w_{\alpha}^p(x) \, \mathrm{d} x \right) \, \mathrm{d} \nu.
	\end{multline}
	The second term of (\ref{378}) using again the triangular inequality and integrating on a greater domain:
	\begin{multline} \label{374}
		C^p h^{rp} \| (F_{\gamma,h,i+1} - F_{\gamma,h,i})^{(r)} w_{\alpha} \varphi^r \|_{L^p([t_i,t_{i+1}])}^p \\
		\leq C^p 2^{p-1}h^{rp} \left( \| (F_{\gamma,h,i+1})^{(r)} w_{\alpha} \varphi^r \|_{L^p([t_i,t_{i+2}])}^p + \| (F_{\gamma,h,i})^{(r)} w_{\alpha} \varphi^r \|_{L^p([t_{i-1},t_{i+1}])}^p \right),
	\end{multline}
	so it is sufficient to estimate the integrals
	\begin{equation*} 
		J_n := h^{rp} \| (F_{\gamma,h,n})^{(r)} \varphi^r w_{\alpha} \|_{L^p([t_{n-1},t_{n+1}])}^p 
	\end{equation*}
	where $n=1,\ldots,j$.
	Using the Hölder inequality we get
	\begin{equation} \label{376}
		J_n \leq \int_{t_{n-1}}^{t_{n+1}} h^{rp} \frac{2a}{h} \int_{\frac{h}{2a}}^{\frac{h}{a}} \left| (f_{\gamma,\tau,\varphi(t_{n-1})})^{(r)}(x) \right|^p \, \mathrm{d} \tau \varphi^{rp}(x) w_{\alpha}^p(x) \, \mathrm{d} x.
	\end{equation}
	From the definition of $f_{\gamma,\tau, \varphi(t_{n-1})}$
	\begin{equation*} 
		(f_{\gamma,\tau,\varphi(t_{n-1})})^{(r)}(x) = r^r \sum_{l=1}^{r} (-1)^{l+1} \binom{r}{l} \frac{1}{(l \gamma^{-1}(\tau) \sqrt{t_{n-1}})^r} \overrightarrow{\Delta}_{\frac{l \gamma^{-1}(\tau) \sqrt{t_{n-1}}}{r}}^{r} f(x)
	\end{equation*}
	holds, so by using the Jensen inequality and this remark we get from (\ref{376})
	\begin{equation} \label{331}
		J_n \leq C^p \sum_{l=1}^{r} \frac{2a}{h} \int_{t_{n-1}}^{t_{n+1}} \int_{\frac{h}{2a}}^{\frac{h}{a}} \frac{h^{rp} \varphi^{rp}(x) w_{\alpha}^{p}(x)}{(\gamma^{-1}(\tau))^{rp} \varphi^{rp}(t_{n-1})} \left| \overrightarrow{\Delta}_{\frac{l \gamma^{-1}(\tau) \sqrt{t_{n-1}}}{r}}^r f(x) \right|^p \, \mathrm{d} \tau \mathrm{d} x.
	\end{equation}
	Since $\frac{h}{2a} \leq \tau \leq \frac{h}{a}$ and $\frac{h}{2a} < h < \gamma(a_1)$, then by (\ref{200})
	\begin{equation} \label{314}
		\frac{h^{rp}}{(\gamma^{-1}(\tau))^{rp}} \leq \frac{h^{rp}}{\left( \gamma^{-1} \left( \frac{h}{2a} \right) \right)^{rp}} \leq \frac{h^{rp}}{\left( \frac{a_1}{\gamma(a_1)} \right)^{rp} \frac{h^{rp}}{(2a)^{rp}}} = \left( \frac{2 a \gamma(a_1)}{a_1} \right)^{rp}.
	\end{equation}
	In addition $t_{n-1} \leq x \leq t_{n+1}$, so by (\ref{77}) $x \sim t_{n-1}$ and $\varphi^{rp}(x) \sim \varphi^{rp}(t_{n-1})$ too.
	So there is a constant $C$, such that we can estimate (\ref{331}):
	\begin{equation} \label{333}
		J_n \leq C^p \sum_{l=1}^{r} \frac{2a}{h} \int_{t_{n-1}}^{t_{n+1}} \int_{0}^{\frac{h}{a}} \left| \overrightarrow{\Delta}_{\frac{l \gamma^{-1}(\tau) \sqrt{t_{n-1}}}{r}}^r f(x) \right|^p w_{\alpha}^{p}(x) \, \mathrm{d} \tau \mathrm{d} x.
	\end{equation}
	In the inner integral let us change the variables $\frac{l \gamma^{-1}(\tau) \sqrt{t_{n-1}}}{r} =: \gamma^{-1}(\nu) \sqrt{x}$. Due to the right choice for $T$, we can do this, and as soon as we procedure from (\ref{76}), we get that
	\begin{equation*} 
		\int_{0}^{\frac{h}{a}} \left| \overrightarrow{\Delta}_{\frac{l \gamma^{-1}(\tau) \sqrt{t_{n-1}}}{r}}^{r} f(x) \right|^p w_{\alpha}^p(x) \, \mathrm{d} \tau \leq C^p \int_{0}^{\frac{h}{a}} \left| \overrightarrow{\Delta}_{\gamma,\varphi,\nu}^{r} f(x) \right|^p w_{\alpha}^p(x) \, \mathrm{d} \nu.
	\end{equation*}
	Substituting to (\ref{333}) and using te Fubini theorem:
	\begin{equation} \label{334}
		J_n \leq 2r C^p \frac{a}{h} \int_{0}^{\frac{h}{a}} \int_{t_{n-1}}^{t_{n+1}} \left| \overrightarrow{\Delta}_{\gamma,\varphi,\nu}^{r} f(x) \right|^p w_{\alpha}^p(x) \, \mathrm{d} x \, \mathrm{d} \nu,
	\end{equation}
	Let us write this into (\ref{374}), then we get by (\ref{378}) and (\ref{242}) that if $1 \leq i \leq j-1$, $k=0,\ldots,r$:
	\begin{equation} \label{381}
		D_{k,i} \leq C^p \frac{a}{h} \int_{0}^{\frac{h}{a}} \left( \int_{t_{i}}^{t_{i+2}} \left| \overrightarrow{\Delta}_{\gamma,\varphi,\nu}^{r} f(x) \right|^p w_{\alpha}^p(x) \, \mathrm{d} x + \int_{t_{i-1}}^{t_{i+1}} \left| \overrightarrow{\Delta}_{\gamma,\varphi,\nu}^{r} f(x) \right|^p w_{\alpha}^p(x) \, \mathrm{d} x \right) \, \mathrm{d} \nu.
	\end{equation}
	The estimation of the first term in (\ref{148}) is yet to come: we do similarly as we did it at the estimation of the $D_{k,i}$-s. 
	If $i=1,\ldots,M+1$, then similarly as in (\ref{167}) $w_{\alpha}(t_{i-1}) \varphi^r(t_{i-1}) \sim w_{\alpha}(x) \varphi^r(x)$,
	and by lemma \ref{37} we have $w_{\alpha}(x) \sim w_{\alpha}(t_i) \sim w_{\alpha}(t_{i-1})$, so we get with a computation as in (\ref{265}) and (\ref{362}) that
	\begin{multline} \label{383}
		h^{rp} \| (F_{\gamma,h,i})^{(r)}_{\gamma} \varphi^{r} w_{\alpha} \|_{L^p([t_i,t_{i+1}])}^{p} 
		\leq  h^{rp} C^p w_{\alpha}^p(t_{i-1}) \varphi^{rp}(t_{i-1}) \| (F_{\gamma,h,i})^{(r)}_{\gamma} \|^{p}_{L^p([t_{i-1},t_{i+1}])} \\
		\leq C^p \sum_{m=1}^{r-1} \underbrace{h^{rp} \| (F_{\gamma,h,i})^{(m)}\varphi^{r} w_{\alpha} \|^p_{L^p([t_{i-1},t_{i+1}])}}_{=: \widetilde{D}_{m,i}} + C^p h^{rp} \| (F_{\gamma,h,i})^{(r)} \varphi^{r} w_{\alpha} \|^p_{L^p([t_{i-1},t_{i+1}])},
	\end{multline}
	if $r=1$, the first sum is empty. Let us suppose that $r \geq 2$, then the last term we can estimate easily by (\ref{334}):
	\begin{equation} \label{415}
		C^p h^{rp} \| (F_{\gamma,h,i})^{(r)} \varphi^{r} w_{\alpha} \|^p_{L^p([t_{i-1},t_{i+1}])} \leq C^p \frac{a}{h} \int_{0}^{\frac{h}{a}} \int_{t_{i-1}}^{t_{i+1}} \left| \overrightarrow{\Delta}_{\gamma,\varphi,\nu}^{r} f(x) \right|^p w_{\alpha}^p(x) \, \mathrm{d} x \, \mathrm{d} \nu.
	\end{equation}
	For the estimation of the first $r-1$ terms let us use that if $1 \leq m \leq r-1$, then
	\begin{multline} \label{417}
		(f_{\gamma,\tau,\varphi(t_{i-1})})^{(m)}(x) = r^r \sum_{l=1}^{r} (-1)^{l+1} \binom{r}{l} \frac{1}{(l \gamma^{-1}(\tau) \sqrt{t_{i-1}})^m} \\
		\times \underbrace{\int_{0}^{\frac{1}{r}} \cdots \int_{0}^{\frac{1}{r}}}_{r-m} \overrightarrow{\Delta}^m_{\frac{l \gamma^{-1}(\tau) \sqrt{t_{i-1}}}{r}} f(x + l \gamma^{-1}(\tau) \sqrt{t_{i-1}}(u_{m+1} + \ldots + u_r)) \, \mathrm{d}u_{m+1} \ldots \mathrm{d}u_r.
	\end{multline}
	Let us introduce the new variable $u := u_{m+1} + \ldots + u_r$, then $0 < u < \frac{r-m}{r}$, and there is a $B^{r,m}$ contant (the volume of the corresponding simplex), thus with the help of the Hölder and Jensen inequalities and (\ref{417}), we get
	\begin{multline} \label{420}
	|(f_{\gamma,\tau,\varphi(t_{i-1})})^{(m)}(x)|^p \leq 
		r^{rp} r^{p-1} \sum_{l=1}^{r} \binom{r}{l}^p  \frac{(B^{r,m})^p}{(l \gamma^{-1}(\tau) \sqrt{t_{i-1}})^{mp}} \left( \frac{r-m}{r} \right)^{mp} \\
		\times \int_{0}^{\frac{r-m}{r}} \left| \overrightarrow{\Delta}^m_{\frac{l \gamma^{-1}(\tau) \sqrt{t_{i-1}}}{r}} f(x + l \gamma^{-1}(\tau) \sqrt{t_{i-1}}u) \right|^p \, \mathrm{d}u \\
		\leq C^p \sum_{l=1}^{r} \frac{1}{(\gamma^{-1}(\tau) \sqrt{t_{i-1}})^{mp}} \int_{0}^{1} \left| \overrightarrow{\Delta}^m_{\frac{l \gamma^{-1}(\tau) \sqrt{t_{i-1}}}{r}} f(x + l \gamma^{-1}(\tau) \sqrt{t_{i-1}}u) \right|^p \, \mathrm{d}u.
	\end{multline}
	By the definition of $F_{\gamma,h,i}$, the Hölder inequality and (\ref{420}) we get that
	\begin{multline*} 
		\widetilde{D}_{m,i} 
		\leq C^p \sum_{l=1}^{r} \Bigg( \frac{h^{rp}}{(\gamma^{-1}(\tau) \sqrt{t_{i-1}})^{mp}} \cdot \frac{2a}{h} \\
		\times \int_{\frac{h}{2a}}^{\frac{h}{a}} \int_{0}^{1} \int_{t_{i-1}}^{t_{i+1}} \left| \overrightarrow{\Delta}^m_{\frac{l \gamma^{-1}(\tau) \sqrt{t_{i-1}}}{r}} f(x + l \gamma^{-1}(\tau) \sqrt{t_{i-1}}u) \right|^p \varphi^{rp}(x) w_{\alpha}^p(x) \, \mathrm{d}x \mathrm{d} u \mathrm{d} \tau \Bigg).
	\end{multline*}
	In the inner integral let us change the variables $x + l \gamma^{-1}(\tau) \sqrt{t_{i-1}} u =: \xi$. Since $\tau \leq \frac{h}{a} < h$, $u < 1$ and $t_{i-1} \leq x$, thus $x \sim \xi$, so
	by lemma \ref{37} we get $w_{\alpha}(x) \sim w_{\alpha}(\xi)$, and
	\begin{multline} \label{422}
		\widetilde{D}_{m,i} 
		\leq C^p \sum_{l=1}^{r} \frac{h^{rp}}{(\gamma^{-1}(\tau) \sqrt{t_{i-1}})^{mp}} \cdot \frac{2a}{h} \\
		\times \int_{\frac{h}{2a}}^{\frac{h}{a}} \int_{0}^{1} \int_{t_{i-1} + l \gamma^{-1}(\tau) \sqrt{t_{i-1}} u}^{t_{i+1} + l \gamma^{-1}(\tau) \sqrt{t_{i-1}} u} \left| \overrightarrow{\Delta}^m_{\frac{l \gamma^{-1}(\tau) \sqrt{t_{i-1}}}{r}} f(\xi) \right|^p \varphi^{rp}(\xi) w_{\alpha}^p(\xi) \, \mathrm{d} \xi \mathrm{d} u \mathrm{d} \tau.
	\end{multline}
	Since $a$ was choosen properly by (\ref{281}), and $\frac{h}{a} < h < \gamma(a_1)$, by using (\ref{200}):
	\begin{equation*} 
		t_{i+1} + l \gamma^{-1}(\tau) \sqrt{t_{i-1}} u < t_{i+1} + r \sqrt{t_{i+1}} \frac{1}{\sum_{k=1}^N \beta_k a_k^{\beta_k - 1}} \cdot \frac{h}{a} < t_{i+2},
	\end{equation*}
	with the help of (\ref{52}). So integrating on a greater domain, the estimation of (\ref{422}) can be carried on:
	\begin{equation} \label{424}
		\widetilde{D}_{m,i} 
		\leq C^p \sum_{l=1}^{r} \frac{2a}{h} \int_{\frac{h}{2a}}^{\frac{h}{a}} \int_{t_{i-1}}^{t_{i+2}} \frac{h^{rp} \varphi^{rp}(\xi)}{(\gamma^{-1}(\tau) \sqrt{t_{i-1}})^{mp}} \left| \overrightarrow{\Delta}^m_{\frac{l \gamma^{-1}(\tau) \sqrt{t_{i-1}}}{r}} f(\xi) \right|^p w_{\alpha}^p(\xi) \, \mathrm{d} \xi \mathrm{d} \tau.
	\end{equation}
	Since $\frac{h}{2a} \leq \tau \leq \frac{h}{a}$ and $\frac{h}{2a} < h < \gamma(a_1)$, similarly as in (\ref{314}) we get that
	\begin{equation*} 
		\frac{h^{mp}}{(\gamma^{-1}(\tau))^{mp}} \leq \left( \frac{2 a \gamma(a_1)}{a_1} \right)^{mp} 
	\end{equation*}
	In addition $\varphi(\xi) \sim \varphi(t_{i}) \sim \varphi(t_{i-1})$, so there exists a constant $C$:
	\begin{equation*} 
		\frac{\varphi^{rp}(\xi)}{\varphi^{mp}(t_{i-1})} \leq C^p \frac{\varphi^{rp}(t_{i-1})}{\varphi^{mp}(t_{i-1})} = C^p \varphi^{(r-m)p}(t_{i-1}) \leq C^p \left( \sqrt{a_N + 1} \right)^{rp},
	\end{equation*}
	because now $1 \leq i \leq M+1$. 
	Using these remarks we can estimate (\ref{424}):
	\begin{equation} \label{429}
		\widetilde{D}_{m,i} \leq 
		C^p \sum_{l=1}^{r} h^{(r-m)p} \int_{t_{i-1}}^{t_{i+2}} \frac{2a}{h} \int_{\frac{h}{2a}}^{\frac{h}{a}} \left| \overrightarrow{\Delta}^m_{\frac{l \gamma^{-1}(\tau) \sqrt{t_{i-1}}}{r}} f(\xi) \right|^p w_{\alpha}^p(\xi) \, \mathrm{d} \tau \mathrm{d} \xi.
	\end{equation}
	In the inner integral let us change the variables $\frac{l \gamma^{-1}(\tau) \sqrt{t_{i-1}}}{r} =: \gamma^{-1}(\nu) \sqrt{\xi}$. We can do this similarly as above, because we chose $T$ properly. Similarly computing as previously, we get from (\ref{429}) that
	\begin{equation*} 
		\widetilde{D}_{m,i} 
		\leq r C^p h^{(r-m)p} \frac{2a}{h} \int_{0}^{\frac{h}{a}} \int_{t_{i-1}}^{t_{i+2}} \left| \overrightarrow{\Delta}_{\gamma,\varphi,\nu}^{m} f(\xi) \right|^p w_{\alpha}^p(\xi) \, \mathrm{d} \xi \mathrm{d} \nu,
	\end{equation*}
	by the Fubini theorem. 
	As $h < t$, the sum of the first $r-1$ terms of (\ref{383}) expression:
	\begin{multline*} 
		C^p \sum_{m=1}^{r-1} \widetilde{D}_{m,i} 
		\leq C^p \sum_{m=1}^{r-1} \frac{a}{h} \int_{a}^{\frac{h}{a}} \int_{t_{i-1}}^{t_{i+2}} |t^{r-m} \overrightarrow{\Delta}_{\gamma,\varphi,\nu}^{m} f(x)|^p w_{\alpha}^p(x) \, \mathrm{d}x \mathrm{d} \nu \\
		\leq C^p \sum_{m=1}^{r-1} \frac{a}{h} \int_{0}^{\frac{h}{a}}\int_{t_{i-1}}^{t_{i+2}} \left( \sum_{n=1}^{r-1} |t^{r-n} \overrightarrow{\Delta}_{\gamma,\varphi,\nu}^{n} f(x)| \right)^p w_{\alpha}^p(x) \, \mathrm{d}x \mathrm{d} \nu \\
		= (r-1) C^p \frac{a}{h} \int_{0}^{\frac{h}{a}} \int_{t_{i-1}}^{t_{i+2}} \left( \sum_{n=1}^{r-1} |t^{r-n} \overrightarrow{\Delta}_{\gamma,\varphi,\nu}^{n} f(x)| \right)^p w_{\alpha}^p(x) \, \mathrm{d}x \mathrm{d}\nu.
	\end{multline*}
	These integrals are well defined, because now $1 \leq i \leq M+1$, so $i+2 \leq M+3 \leq j+1$.
	Looking to (\ref{415}) we get that if $i=1,\ldots,M+1$, then
	\begin{multline} \label{444}
		h^{rp} \| (F_{\gamma,h,i})^{(r)}_{\gamma} \varphi^{r} w_{\alpha} \|_{L^p([t_i,t_{i+1}])}^{p} \leq  h^{rp} \| (F_{\gamma,h,i})^{(r)}_{\gamma} \varphi^{r} w_{\alpha} \|_{L^p([t_{i-1},t_{i+1}])}^{p} \\
		\leq 2 C^p \frac{a}{h} \int_{0}^{\frac{h}{a}} \int_{t_{i-1}}^{t_{i+2}} \left( \sum_{n=1}^{r} |t^{r-n} \overrightarrow{\Delta}_{\gamma,\varphi,\nu}^{n} f(x)| \right)^p w_{\alpha}^p(x) \, \mathrm{d}x \mathrm{d}\nu,
	\end{multline}
	this is valid for $r=1$ too.

	Finally if $i=M+2,\ldots,j-1$, then by the statement of \ref{337} simply
	\begin{multline} \label{384}
		h^{rp} \| (F_{\gamma,h,i})^{(r)}_{\gamma} \varphi^{r} w_{\alpha} \|_{L^p([t_i,t_{i+1}])}^{p} \leq h^{rp} \| (F_{\gamma,h,i})^{(r)}_{\gamma} \varphi^{r} w_{\alpha} \|_{L^p([t_{i-1},t_{i+1}])}^{p} 
		\leq \frac{C^p}{\mathcal{C}_1^{rp}} J_{i} \\
		\leq C^p \frac{a}{h} \int_{0}^{\frac{h}{a}} \int_{t_{i-1}}^{t_{i+1}} \left( \sum_{n=1}^{r} |t^{r-n} \overrightarrow{\Delta}_{\gamma,\varphi,\nu}^{n} f(x)| \right)^p w_{\alpha}^p(x) \, \mathrm{d} x \mathrm{d} \nu.
	\end{multline}
	Based on this using (\ref{148}) and (\ref{381}) we get that if $i=1,\ldots,j-1$, then
	\begin{equation*} 
		h^{rp} \| (G_{\gamma,h})^{(r)}_{\gamma} \varphi^r w_{\alpha} \|_{L^p([t_i,t_{i+1}])}^{p} 
		\leq 3C^p \frac{a}{h} \int_{0}^{\frac{h}{a}} \int_{t_{i-1}}^{t_{i+2}} \left( \sum_{n=1}^{r} |t^{r-n} \overrightarrow{\Delta}_{\gamma,\varphi,\nu}^{n} f(x)| \right)^p w_{\alpha}^p(x) \, \mathrm{d} x  \mathrm{d} \nu.
	\end{equation*}

	If $i=0$, then by (\ref{164}), (\ref{383}) and estimation (\ref{444})
	\begin{equation*} 
		h^{rp} \| (G_{\gamma,h})^{(r)}_{\gamma} \varphi^r w_{\alpha} \|_{L^p([t_0,t_{1}])}^{p} 
		\leq 2 C^p \frac{a}{h} \int_{0}^{\frac{h}{a}} \int_{t_{0}}^{t_{3}} \left( \sum_{n=1}^{r} |t^{r-n} \overrightarrow{\Delta}_{\gamma,\varphi,\nu}^{n} f(x)| \right)^p w_{\alpha}^p(x) \, \mathrm{d} x \mathrm{d} \nu,
	\end{equation*}
	and if $i=j$, then by (\ref{165})
	\begin{equation*} 
		h^{rp} \| (G_{\gamma,h})^{(r)}_{\gamma} \varphi^r w_{\alpha} \|_{L^p([t_j,t_{j+1}])}^{p} 
		\leq C^p \frac{a}{h} \int_{0}^{\frac{h}{a}} \int_{t_{j-1}}^{t_{j+1}} \left( \sum_{n=1}^{r} |t^{r-n} \overrightarrow{\Delta}_{\gamma,\varphi,\nu}^{n} f(x)| \right)^p w_{\alpha}^p(x) \, \mathrm{d} x \mathrm{d} \nu,
	\end{equation*}
	because the estimation (\ref{384}) is true in the case $i=j$ too. So let us return to (\ref{147}), then we get that
	\begin{multline*} 
		h^{rp} \| (G_{\gamma,h})^{(r)}_{\gamma} \varphi^r w_{\alpha} \|_{L^p(I_{rh, \gamma})}^{p} 
		\leq 4 C^p \frac{a}{h} \int_{0}^{\frac{h}{a}} \int_{t_0}^{t_{j+1}} \left( \sum_{n=1}^{r} |t^{r-n} \overrightarrow{\Delta}_{\gamma,\varphi,\nu}^{n} f(x)| \right)^p w_{\alpha}^p(x) \, \mathrm{d} x \mathrm{d}\nu \\
		\leq C^p \frac{a}{h} \int_{0}^{\frac{h}{a}} \left\| w_{\alpha} \sum_{n=1}^{r} |t^{r-n} \overrightarrow{\Delta}_{\gamma,\varphi,\nu}^{n} f| \right\|_{L^p(I_{rh,\gamma}')}^p \mathrm{d}\nu 
		\leq C^p \sup_{0 < h \leq t} \left\| w_{\alpha} \sum_{n=1}^{r} |t^{r-n} \overrightarrow{\Delta}_{\gamma,\varphi,h}^{n} f| \right\|_{L^p(I_{rh,\gamma}')}^p.
	\end{multline*}
	From this we get that for all $0 < h \leq t$ the inequality (\ref{144}) holds, which we wanted to prove.
	Since the constant in the estimations is independent of $p$, (\ref{144}) 
	is also true if $p=\infty$.
	\end{lepes}

	Finally by definiton (\ref{14}), by (\ref{62}) and (\ref{144}) for all $1 \leq p \leq \infty$
	\begin{equation*} 
		\widetilde{K}_{\gamma,r,\varphi}(f,t^r,L^p_{w_{\alpha}},W_{\gamma}^{r,p}) 
		\leq C \Omega_{\gamma,\varphi}^{r} (f,t)_{w_{\alpha},p} + C \sum_{n=1}^{r} t^{r-n} \Omega_{\gamma,\varphi}^{n} (f,t)_{w_{\alpha},p}, 
	\end{equation*}
	from which the inequality (\ref{51}) follows.

	
\begin{lepes} \label{576}	In the further section of the proof we get a lower bound for $\widetilde{K}_{\gamma,r,\varphi}(f,t^r,L^p_{w_{\alpha}},W_{\gamma}^{r,p})$, which proves inequality (\ref{391}). We investigate also the $I_{rh,\gamma}$ interval,
where now $\frac{1}{4} < A_1$, $0 < A_2 < \frac{a_1^2}{4}$. Let
	\begin{equation} \label{358a}
		0 < h \leq t < T < \min \left\{ \textstyle \gamma \left( \sqrt[4]{\frac{A_2}{A_1}} \frac{1}{\sqrt{2r}} \right), \gamma(a_1) \displaystyle \right\}.
	\end{equation}
	At first let $1 < p < \infty$, and let $g \in W_{\gamma^{r,p}}(I_{rh,\gamma}')$ a function such that
	\begin{equation*} 
		\sup_{0 < h \leq t} \left\{ \| (f-g) w_{\alpha} \|_{L^p(I_{rh,\gamma}')} + h^r \| g^{(r)}_{\gamma} \varphi^{r} w_{\alpha} \|_{L^p(I_{rh,\gamma}')} \right\} \leq 2 \widetilde{K}_{\gamma,r,\varphi}(f,t^r,L^p_{w_{\alpha}},W_{\gamma}^{r,p}).
	\end{equation*}
	Then $g \in W_{\gamma^{r,p}}(I_{rh,\gamma})$ is also true, because $I_{rh,\gamma} \subset I_{rh,\gamma}'$. By the linearity of operator $\overrightarrow{\Delta}_{\gamma,\varphi,h}^r$ and the triangle inequality
	\begin{equation} \label{357a}
		\| w_{\alpha} \overrightarrow{\Delta}_{\gamma,\varphi,h}^r f \|_{L^p(I_{rh,\gamma})} \leq \underbrace{\| w_{\alpha} \overrightarrow{\Delta}_{\gamma,\varphi,h}^r (f-g) \|_{L^p(I_{rh,\gamma})}}_{N_1} + \underbrace{\| w_{\alpha} \overrightarrow{\Delta}_{\gamma,\varphi,h}^r g \|_{L^p(I_{rh,\gamma})}}_{N_2}.
	\end{equation}
	Let us estimate $N_1$: by using the Jensen inequality and lemma \ref{37}:
	\begin{equation*} 
		N_1^p 
		\leq C^p \sum_{k=0}^{r} \int_{4A_1 r^2 (\gamma^{-1}(h))^2}^{\frac{A_2}{(\gamma^{-1}(h))^2}} \left| (f-g) w_{\alpha} (x+k \gamma^{-1}(h) \sqrt{x}) \right|^p \, \mathrm{d} x.
	\end{equation*}
	Let us change the variables $x + k \gamma^{-1}(h) \sqrt{x} =: y$, then
	\begin{equation} \label{395}
		N_1^p 
		< C^p \sum_{k=0}^{r} \int_{4A_1 r^2 (\gamma^{-1}(h))^2}^{\frac{A_2 \frac{1 + 2 \sqrt{A_1}}{2 \sqrt{A_1}}}{(\gamma^{-1}(h))^2}} w_{\alpha}^p(y) |(f-g)(y)|^p \, \mathrm{d}y = C^p r \| (f-g) w_{\alpha} \|_{L^p(I_{rh,\gamma}')}^p, 
	\end{equation}
	where in the last step we used that we integrated on a greater domain, because of the appropriate choosing of $h$ $k=0,\ldots,r$
	\begin{equation*} 
		\textstyle \frac{A_2}{(\gamma^{-1}(h))^2} + k \sqrt{A_2} \leq \frac{A_2 \left( 1 + r \frac{1}{\sqrt{A_2}} (\gamma^{-1}(h))^2 \right)}{(\gamma^{-1}(h))^2} \leq \frac{A_2 \left( 1 + \frac{1}{2 \sqrt{A_1}} \right)}{(\gamma^{-1}(h))^2}.
	\end{equation*}
	hold for all $k=0,\ldots,r$. For the estimation of $N_2$ let us use the following identity, which can be proved with induction:
	\begin{equation} \label{397}
		\overrightarrow{\Delta}_{\gamma,\varphi,h}^r g(x) 
		= \int_{0}^{\gamma^{-1}(h) \varphi(x)} \cdots \int_{0}^{\gamma^{-1}(h) \varphi(x)} g^{(r)}_{\gamma}(x + u_1 + \ldots + u_r) \prod_{i=1}^{r} \gamma'(u_i) \, \mathrm{d} u_1 \ldots \mathrm{d}u_r.
	\end{equation}
	This latter integral is well defined, because now $x \in I_{rh,\gamma}$, and for all of the integrating variables
	\begin{equation*} 
		\textstyle 0 < u_i < \gamma^{-1}(h) \sqrt{x} \leq \gamma^{-1}(h) \frac{A_2}{\gamma^{-1}(h)} = \sqrt{A_2} < \frac{a_1}{2},
	\end{equation*}
	and $\gamma'(u_i)$ is finite on $\left[ 0,\frac{a_1}{2} \right]$. 
	So we get from (\ref{397}) that
	\begin{equation*} 
		| \overrightarrow{\Delta}_{\gamma,\varphi,h}^r g(x) | \leq 
		C \int_{0}^{\gamma^{-1}(h) \varphi(x)} \cdots \int_{0}^{\gamma^{-1}(h) \varphi(x)} | g^{(r)}_{\gamma}(x + u_1 + \ldots + u_r) | \, \mathrm{d} u_1 \ldots \mathrm{d}u_r.
	\end{equation*}
	Let $u = u_1 + \ldots + u_r$, then since
	\begin{multline*} 
		\{ (u_1,\ldots,u_r) \in \mathds{R}^r \mid 0 < u_i < \gamma^{-1}(h) \varphi(x), i = 1,\ldots,r \} \\
		\subset \{ (u_1,\ldots,u_r) \in \mathds{R}^r \mid 0 < u_1 + \ldots + u_r < r \gamma^{-1}(h) \varphi(x), 0 < u_i, i=1,\ldots,r \} =: H,
	\end{multline*}
	we get
	\begin{multline} \label{403}
		| \overrightarrow{\Delta}_{\gamma,\varphi,h}^r g(x) | 
		\leq C \int_H  | g^{(r)}_{\gamma}(x + u_1 + \ldots + u_r) | \, \mathrm{d} u_1 \ldots \mathrm{d}u_r \\ 
		\leq C \lambda(H) \int_{0}^{r \gamma^{-1}(h) \varphi(x)} |g^{(r)}_{\gamma} ( x+u )| \mathrm{d}u 
		= \frac{C (r \gamma^{-1}(h) \varphi(x))^r}{r!} \int_{0}^{r \gamma^{-1}(h) \varphi(x)} |g^{(r)}_{\gamma} ( x+u )| \mathrm{d}u,
	\end{multline}
	because $H$ is an $r$-dimensional simplex, and its volume is $\lambda(H) = \frac{(r \gamma^{-1}(h) \varphi(x))^r}{r!}$. Since
	\begin{equation} \label{404}
		(\gamma^{-1}(h))^r (\varphi(x))^r 
		\leq (\gamma^{-1}(h))^{r-1} (\varphi(x))^{r-1} \gamma^{-1}(h) \textstyle \frac{\sqrt{A_2}}{\gamma^{-1}(h)} = \sqrt{A_2} (\gamma^{-1}(h))^{r-1} (\varphi(x))^{r-1},
	\end{equation}
	by 
	(\ref{403}) there exists a constant $C$ such that
	\begin{equation} \label{405}
		| \overrightarrow{\Delta}_{\gamma,\varphi,h}^r g(x) | \leq C (\gamma^{-1}(h))^{r-1} (\varphi(x))^{r-1} \int_{0}^{r \gamma^{-1}(h) \varphi(x)} |g^{(r)}_{\gamma} ( x+u )| \mathrm{d}u.
	\end{equation}
	Moreover $0 < u < r \gamma^{-1}(h) \sqrt{x}$, thus $x < x + u < x + r \gamma^{-1}(h) \sqrt{x}$,
	from which we get $x \sim x+ u$. Furthermore $x \in I_{rh,\gamma}$, so by lemma \ref{37} $w_{\alpha}(x) \sim w_{\alpha}(x+u)$. Similarly, since $x \sim x + u$, $\varphi(x) \sim \varphi(x+u)$ is also true. With these remarks, changing the variables $u + x =: y$, we get from the definition of $N_2$ and (\ref{405}):
	\begin{equation*} 
	N_2	\leq C (\gamma^{-1}(h))^r \left(  \int_{4 A_1 r^2 (\gamma^{-1}(h))^2}^{\frac{A_2}{(\gamma^{-1}(h))^2}} \left( \textstyle \frac{1}{r \gamma^{-1}(h) \varphi(x)} \displaystyle \int_{x}^{x+ r \gamma^{-1}(h) \varphi(x)} |w_{\alpha}(y) g^{(r)}_{\gamma} (y)| \varphi^{r}(y) \mathrm{d}y \right)^p \, \mathrm{d}x  \right)^{\frac{1}{p}}.
	\end{equation*}
	In the resulting expression let us use 
	the Hardy--Littlewood maximal inequality for the $w_{\alpha} g^{(r)}_{\gamma} \varphi^r \chi_{I_{rh,\gamma}}$ function, and we get that
	\begin{equation} \label{412}
		N_2 \leq C(p) h^{r} \| w_{\alpha} g^{(r)}_{\gamma} \varphi^r \|_{L^p(I_{rh,\gamma}')},
	\end{equation}
	because $h < \gamma(a_1)$, so we can use the estimation (\ref{200}), and $I_{rh,\gamma} \subset I_{rh,\gamma}'$. So by (\ref{357a}), (\ref{395}) and (\ref{412}) we get that for all $0 < h \leq t$
	\begin{multline*} 
		\| w_{\alpha} \overrightarrow{\Delta}_{\gamma,\varphi,h}^{r} f \|_{L^p(I_{rh,\gamma})} \leq C \left( \| (f-g) w_{\alpha} \|_{L^p(I_{rh,\gamma}')} + h^r \| \varphi^r g^{(r)}_{\gamma} w_{\alpha} \|_{L^p(I_{rh,\gamma}')} \right) \\
		\leq C \sup_{0 < h \leq t} \left\{ \| (f-g) w_{\alpha} \|_{L^p(I_{rh,\gamma}')} + h^r \| \varphi^r g^{(r)}_{\gamma} w_{\alpha} \|_{L^p(I_{rh,\gamma}')} \right\} \\
		\leq C \widetilde{K}_{\gamma,r,\varphi}(f,t^r,L^p_{w_{\alpha}},W_{\gamma}^{r,p}).
	\end{multline*}
	Since this estimation is valid for all $0 < h \leq t$ we get the inequality (\ref{391}), if $1 < p < \infty$, which we wanted to prove.
If $p = 1$, then we can do similarly as above, but this case is simpler: we dont need (\ref{404}) transformation and we can prove the statement without the Hardy--Littlewood maximal intequality.
\end{lepes}
	
	If $p = \infty$, then with the corresponding modifications of the norm we get the statement (\ref{391}), because the Hardy--Littlewood maximal inequality is true for $p = \infty$ too.
	\end{proof}

\begin{proof}[Proof of theorem \ref{481}]
	We prove also in several steps. 
	\begin{lepess} \label{580}
		We get an upper bound for $K_{\gamma,r,\varphi}(f,t^r)$. Let $A_1,A_2$ constants denote the same as in the first section of the proof of theorem \ref{50} (when we proved (\ref{51})), and let $0 < t \leq T$ be small enough so that it satisfies assumption (\ref{250}), and the following:
	\begin{equation*} 
		T < \min \left\{ \textstyle \gamma(a_1), \gamma \left( \sqrt{\frac{a_1}{4 r^2 \sqrt{A_1} (1 + 2 \sqrt{A_1})}} \right) \displaystyle \right\}.
	\end{equation*}
	Let as previously
	\begin{equation*} 
		I_{rt,\gamma} = \left[4A_1r^2 (\gamma^{-1}(t))^2, \frac{A_2}{(\gamma^{-1}(t))^{2}} \right],
	\end{equation*}
	end let us consider the system of points $t_0,\ldots,t_j,t_{j+1}$ such that
	\begin{equation*} 
		t_0 := 4A_1r^2 (\gamma^{-1}(t))^2 < t_1 < t_2 < \ldots t_j < \frac{A_2}{(\gamma^{-1}(t))^{2}} \leq t_{j+1} =: \frac{A_2A}{(\gamma^{-1}(t))^{2}},
	\end{equation*}
	and
	\begin{equation} \label{486}
		\frac{1}{2 r} \leq \frac{t_{i+1} - t_i}{\gamma^{-1}(t) \sqrt{t_i}} \leq r, \ \ \ i = 0, \ldots, j.
	\end{equation}
	$I_{rt,\gamma}'' := [t_0,t_{j+1}]$ is an $I_{rt,\gamma}$-type interval (we can say the same thing as above with the partition with $h$). We will show that there exists a constant $C$ such that
	\begin{multline} \label{487}
		K_{\gamma,r,\varphi}(f,t^r,L^p_{w_{\alpha}},W_{\gamma}^{r,p}) \\
		\leq C K_{\gamma,r,\varphi}(f,t^r,L^p_{w_{\alpha}}(0,4 A_1' r^2 (\gamma^{-1}(t))^2),W_{\gamma}^{r,p}(0,4 A_1' r^2 (\gamma^{-1}(t))^2)) \\
		+ C K_{\gamma,r,\varphi}(f,t^r,L^p_{w_{\alpha}}(I_{rt,\gamma}''),W_{\gamma}^{r,p}(I_{rt,\gamma}'')) \\
		+ C K_{\gamma,r,\varphi} \left(f,t^r,L^p_{w_{\alpha}}\left( \textstyle \frac{A_2'}{(\gamma^{-1}(t))^2},\infty \right),W_{\gamma}^{r,p}\left( \textstyle \frac{A_2'}{(\gamma^{-1}(t))^2},\infty \right) \right),
	\end{multline}
	where $A_1 = \frac{1}{2} \sqrt{A_1} (1 + 2 \sqrt{A_1})$ and $A_2' = \frac{2 \sqrt{A_1} A_2}{1 + 2 \sqrt{A_1}}$. For this let $G_{\gamma,t}$ be the function given by (\ref{61}) (we repeat the former constuction with $h = t$ now). Furthermore let $g_1 \in W_{\gamma}^{r,p}(0,t_1)$, $g_2 \in W_{\gamma}^{r,p}(t_0,t_{j+1})$, $g_3 \in W_{\gamma}^{r,p}(t_j,\infty)$ be arbitrary functions, and let
	\begin{multline} \label{488}
		\textstyle g(x) := \left( 1 - \psi \left( \frac{\gamma(x) - \gamma(t_0)}{\gamma(t_1) - \gamma(t_0)} \right) \right) g_1(x) \\
		\textstyle + \psi \left( \frac{\gamma(x) - \gamma(t_0)}{\gamma(t_1) - \gamma(t_0)} \right) \left( 1 - \psi \left( \frac{\gamma(x) - \gamma(t_j)}{\gamma(t_{j+1}) - \gamma(t_j)} \right) \right) g_2(x) 
		\textstyle + \psi \left( \frac{\gamma(x) - \gamma(t_j)}{\gamma(t_{j+1}) - \gamma(t_j)} \right) g_3(x),
	\end{multline}
	where $\psi$ is the function given by (\ref{56}). 
	For this $g$ function:
	\begin{multline} \label{490}
		\| (f-g) w_{\alpha} \|_{L^p} 
		\leq \| (f-g) w_{\alpha} \|_{L^p(0,t_j)} + \| (f-g) w_{\alpha} \|_{L^p(t_1,\infty)} \\
		\leq 2 \left( \| (f-g_1) w_{\alpha} \|_{L^p(0,t_1)} + \| (f-g_2) w_{\alpha} \|_{L^p(t_0,t_{j+1})} + \| (f-g_3) w_{\alpha} \|_{L^p(t_j,\infty)} \right).
	\end{multline}
	Similarly,
	\begin{multline} \label{491}
		t^r \| g^{(r)}_{\gamma} \varphi^{r} w_{\alpha} \|_{L^p} 
		\leq t^r \| (g_1)^{(r)}_{\gamma} \varphi^{r} w_{\alpha} \|_{L^p(0,t_1)} + t^r \| (g_2)^{(r)}_{\gamma} \varphi^{r} w_{\alpha} \|_{L^p(t_0,t_{j+1})} \\
		+ t^r \| (g_3)^{(r)}_{\gamma} \varphi^{r} w_{\alpha} \|_{L^p(t_j,\infty)} + t^r \| g^{(r)}_{\gamma} \varphi^{r} w_{\alpha} \|_{L^p(t_0,t_1)} + t^r \| g^{(r)}_{\gamma} \varphi^{r} w_{\alpha} \|_{L^p(t_j,t_{j+1})}.
	\end{multline}
	The last two terms can be estimated with the norms of the $\gamma$-relative derivatives of $g_1,g_2$ and $g_3$ too, we do this for example with the first term, the estimation of the second is analogous. If $x \in [t_0,t_1]$, then from (\ref{488}) we get
	\begin{equation*} 
		g(x) = g_1(x) + \psi \left( \frac{\gamma(x) - \gamma(t_0)}{\gamma(t_1) - \gamma(t_0)} \right) (g_2(x) - g_1(x)),
	\end{equation*}
	so by statement \ref{140} and statement \ref{25}:
	\begin{equation*} 
		g_{\gamma}^{(r)}(x) = (g_1)_{\gamma}^{(r)}(x) + \sum_{k=0}^{r} \binom{r}{k} \frac{(g_2)^{(k)}_{\gamma}(x) - (g_1)^{(k)}_{\gamma}(x)}{(\gamma(t_1) - \gamma(t_0))^{r-k}} \psi^{(r-k)} \left( \frac{\gamma(x) - \gamma(t_0)}{\gamma(t_1) - \gamma(t_0)} \right),
	\end{equation*}
	from which by (\ref{264}) we get
	\begin{multline} \label{494}
		t^r \| g^{(r)}_{\gamma} \varphi^{r} w_{\alpha} \|_{L^p(t_0,t_1)} \leq t^r \| (g_1)^{(r)}_{\gamma} \varphi^{r} w_{\alpha} \|_{L^p(t_0,t_1)} \\
		+ C \sum_{k=0}^{r} \binom{r}{k} \frac{t^r}{(\gamma(t_1) - \gamma(t_0))^{r-k}} \| (g_2 - g_1)^{(k)}_{\gamma} \varphi^{r} w_{\alpha} \|_{L^p(t_0,t_1)}.
	\end{multline}
	On $[t_0,t_1]$ $w_{\alpha}(x) \varphi^r(x) \sim w_{\alpha}(t_0) \varphi^r(t_0)$ holds, so
	\begin{equation} \label{495}
		\| (g_2 - g_1)^{(k)}_{\gamma} \varphi^{r} w_{\alpha} \|_{L^p(t_0,t_1)} \leq C w_{\alpha}(t_0) \varphi^r(t_0) \| (g_2 - g_1)^{(k)}_{\gamma} \|_{L^p(t_0,t_1)}.
	\end{equation}
	At first let $1 \leq p < \infty$: with the substitution $\gamma(x) =: y$ with the help of the statement \ref{17}:
	\begin{equation} \label{496}
		\| (g_2 - g_1)^{(k)}_{\gamma} \|_{L^p(t_0,t_1)}^p = \int_{\gamma(t_0)}^{\gamma(t_1)} |((g_2 - g_1) \circ \gamma^{-1})^{(k)}(y)|^p (\gamma^{-1})'(y) \, \mathrm{d}y.
	\end{equation}
	We could do this transformation because of the proper choosing of $t$ the derivative $(\gamma^{-1}(y))'$ is finite. 
	So by (\ref{496}) there is a constant $C$ such that
	\begin{equation} \label{501}
		\| (g_2 - g_1)^{(k)}_{\gamma} \|_{L^p(t_0,t_1)} \leq C \| ((g_2 - g_1) \circ \gamma^{-1})^{(k)} \|_{L^p(\gamma(t_0),\gamma(t_1))}.
	\end{equation}
	For the given expression let us use the lemma \ref{359}, using (\ref{495}) and (\ref{501}) estimations we get, that in the expression (\ref{494})
	\begin{multline} \label{503}
		\sum_{k=0}^{r} \binom{r}{k} \frac{t^r}{(\gamma(t_1) - \gamma(t_0))^{r-k}} \| (g_2 - g_1)^{(k)}_{\gamma} \varphi^{r} w_{\alpha} \|_{L^p(t_0,t_1)} \\
		\leq C \frac{t^r w_{\alpha}(t_0) \varphi^{r}(t_0)}{(\gamma(t_1) - \gamma(t_0))^{r}} \| (g_2 - g_1) \circ \gamma^{-1} \|_{L^p(\gamma(t_0),\gamma(t_1))} \\
		+ C t^r w_{\alpha}(t_0) \varphi^{r}(t_0) \| ((g_2 - g_1) \circ \gamma^{-1})^{(r)} \|_{L^p(\gamma(t_0),\gamma(t_1))}.
	\end{multline}
	In the integral of the first term let us change the variables: $x = \gamma^{-1}(y)$, then
	\begin{equation} \label{504}
		\| (g_2 - g_1) \circ \gamma^{-1} \|_{L^p(\gamma(t_0),\gamma(t_1))}^p = \int_{t_0}^{t_1} |(g_2 - g_1)(x)|^p \gamma'(x) \, \mathrm{d}x.
	\end{equation}
	We could do this substitution, because we chose $t$ suitable, furthermore due to the fine choosing we get $t_0 < x < t_1 < \frac{a_1}{2}$, and $\gamma'$ is finite on $\left[ 0, \frac{a_1}{2} \right]$. So there is a constant $C$ such that by (\ref{504}): 
	\begin{equation} \label{506}
		\| (g_2 - g_1) \circ \gamma^{-1} \|_{L^p(\gamma(t_0),\gamma(t_1))} \leq C \| g_2 - g_1 \|_{L^p(t_0,t_1)}.
	\end{equation}
	Similarly, in the integral of the second term in (\ref{503}) after the application of the statement \ref{17} let us change the variables $x := \gamma^{-1}(y)$:
	\begin{equation*} 
		\| ((g_2 - g_1) \circ \gamma^{-1})^{(r)} \|_{L^p(\gamma(t_0),\gamma(t_1))}^p = \int_{t_0}^{t_1} |(g_2 - g_1)^{(r)}_{\gamma}(x)|^p \, \mathrm{d} \gamma(x),
	\end{equation*}
	so
	\begin{equation} \label{508}
		\| ((g_2 - g_1) \circ \gamma^{-1})^{(r)} \|_{L^p(\gamma(t_0),\gamma(t_1))} \leq C \| (g_2 - g_1)^{(r)}_{\gamma} \|_{L^p(t_0,t_1)}.
	\end{equation}
	Furthermore since $\gamma^{-1}$ has the Lipschitz property on $[t_0,t_1]$ and $t < \gamma(a_1)$, we get by (\ref{200}) and inequality (\ref{486}) that
	\begin{multline} \label{509}
		\frac{t^r \varphi^r(t_0)}{(\gamma(t_1) - \gamma(t_0))^r} \leq \underbrace{\left( \frac{t_1 - t_0}{\gamma(t_1) - \gamma(t_0)} \right)^r}_{\text{bounded}} \cdot \left( \frac{t \sqrt{t_0}}{t_1 - t_0} \right)^r \\
		\leq C \left( \frac{\gamma(a_1)}{a_1} \right)^r \left( \frac{\gamma^{-1}(t) \sqrt{t_0}}{t_1 - t_0} \right)^{r} \leq C \left( \frac{\gamma(a_1) 2r}{a_1} \right)^r.
	\end{multline}
	Since if $x \in [t_0,t_1]$, $w_{\alpha}(t_0) \sim w_{\alpha}(x)$ and $w_{\alpha}(x) \varphi^r(x) \sim w_{\alpha}(t_0) \varphi^r(t_0)$, thus we get for the estimation of (\ref{503}) by (\ref{506}), (\ref{508}) and (\ref{509}), that
	\begin{multline*} 
		\sum_{k=0}^{r} \binom{r}{k} \frac{t^r}{(\gamma(t_1) - \gamma(t_0))^{r-k}} \| (g_2 - g_1)^{(k)}_{\gamma} \varphi^{r} w_{\alpha} \|_{L^p(t_0,t_1)} \\
		\leq C \| (g_2 - g_1) w_{\alpha} \|_{L^p(t_0,t_1)} + C t^r \| (g_2 - g_1)^{(r)}_{\gamma} \varphi^r w_{\alpha} \|_{L^p(t_0,t_1)}.
	\end{multline*}
	Returning to (\ref{494}) we get that
	\begin{multline*} 
			t^r \| g^{(r)}_{\gamma} \varphi^{r} w_{\alpha} \|_{L^p(t_0,t_1)} \\ 
			\leq t^r \| (g_1)^{(r)}_{\gamma} \varphi^{r} w_{\alpha} \|_{L^p(0,t_1)} + C \left( \| (f-g_1) w_{\alpha} \|_{L^p(t_0,t_1)} + \| (f-g_2) \|_{L^p(t_0,t_1)} \right) \\
			+ C \left( t^r \| (g_2)^{(r)}_{\gamma} \varphi^r w_{\alpha} \|_{L^p(t_0,t_1)} + t^r \| (g_1)^{(r)}_{\gamma} \varphi^r w_{\alpha} \|_{L^p(t_0,t_1)} \right) \\
			\leq C \left( \| (f-g_1) w_{\alpha} \|_{L^p(0,t_1)} + t^r \| (g_1)^{(r)}_{\gamma} \varphi^r w_{\alpha} \|_{L^p(0,t_1)} \right) \\
			+ C \left( \| (f-g_2) \|_{L^p(t_0,t_{j+1})} + t^r \| (g_2)^{(r)}_{\gamma} \varphi^r w_{\alpha} \|_{L^p(t_0,t_{j+1})} \right).
	\end{multline*}
	Similarly we get that
	\begin{multline*} 
		t^r \| g^{(r)}_{\gamma} \varphi^{r} w_{\alpha} \|_{L^p(t_j,t_{j+1})} \leq C \left( \| (f-g_2) w_{\alpha} \|_{L^p(t_0,t_{j+1})} + t^r \| (g_2)^{(r)}_{\gamma} \varphi^r w_{\alpha} \|_{L^p(t_0,t_{j+1})} \right) \\
			+ C \left( \| (f-g_3) \|_{L^p(t_j,\infty)} + t^r \| (g_3)^{(r)}_{\gamma} \varphi^r w_{\alpha} \|_{L^p(t_j,\infty)} \right).
	\end{multline*}
	Let us substitute these estimates to (\ref{491}), then the given inequlatities let us add to (\ref{490}):
	\begin{multline*} 
		\| (f-g)w_{\alpha} \|_{L^p} + t^r \| g_{\gamma}^{(r)} \varphi^r w_{\alpha} \|_{L^p} \\
		\leq C \left( \| (f-g_1) w_{\alpha} \|_{L^p(0,t_1)} + t^r \| (g_1)^{(r)}_{\gamma} \varphi^r w_{\alpha} \|_{L^p(0,t_1)} \right) \\
			+ C \left( \| (f-g_2) \|_{L^p(t_0,t_{j+1})} + t^r \| (g_2)^{(r)}_{\gamma} \varphi^r w_{\alpha} \|_{L^p(t_0,t_{j+1})} \right) \\
			+ C \left( \| (f-g_3) \|_{L^p(t_j,\infty)} + t^r \| (g_3)^{(r)}_{\gamma} \varphi^r w_{\alpha} \|_{L^p(t_j,\infty)} \right).
	\end{multline*}
	As $g_1,g_2,g_3$ are arbitrary functions in $W_{\gamma}^{r,p}$ on the corresponding interval, by the definition of the $K$-functional we get from this that
	\begin{multline} \label{515}
		K_{\gamma,r,\varphi}(f,t^r,L^p_{w_{\alpha}},W_{\gamma}^{r,p}) 
		\leq C K_{\gamma,r,\varphi}(f,t^r,L^p_{w_{\alpha}}(0,t_1),W_{\gamma}^{r,p}(0,t_1)) \\
		+ C K_{\gamma,r,\varphi}(f,t^r,L^p_{w_{\alpha}}(I_{rt,\gamma}''),W_{\gamma}^{r,p}(I_{rt,\gamma}'')) \\
		+ C K_{\gamma,r,\varphi} \left(f,t^r,L^p_{w_{\alpha}} \left(t_j,\infty \right),W_{\gamma}^{r,p}\left( t_j,\infty \right) \right).
	\end{multline}
	However, due to the fine choosing of $t$ it is easy to check that $t_1 \leq 4 A_1' r^2 (\gamma^{-1}(t))^2$ and $t_j \geq \frac{A_2'}{(\gamma^{-1}(t))^2}$, so we get by these remarks from (\ref{515}) exactly the inequality (\ref{487}).
	We can estimate the middle term of this with the first part of the proved theorem \ref{50}, because $K \leq \widetilde{K}$:
	\begin{equation} \label{520}
		K_{\gamma,r,\varphi}(f,t^r,L^p_{w_{\alpha}}(I_{rt,\gamma}''),W_{\gamma}^{r,p}(I_{rt,\gamma}'')) \leq C \sum_{n=1}^{r} t^{r-n} \Omega_{\gamma,\varphi}^n (f,t)_{w_{\alpha},p},
	\end{equation}
	where these latter expressions we define with the help of interval $I_{rh,\gamma}'$.
	
	On the other hand, if $p \in P_{r-1}$, where $P_{r-1}$ is the space of the polynomials of degree at most $r$, then $p \circ \gamma \in W_{\gamma}^{r,p}$, because (\ref{452}) holds, no matter on which interval we look at the norm. So
	\begin{multline} \label{521}
		K_{\gamma,r,\varphi}(f,t^r,L^p_{w_{\alpha}}(0,4A_1' r^2 (\gamma^{-1}(t))^2),W_{\gamma}^{r,p}(0,4A_1' r^2 (\gamma^{-1}(t))^2)) \\
		\leq \inf_{p \in P_{r-1}} \| (f - p \circ \gamma) w_{\alpha} \|_{L^p(0,4A_1' r^2 (\gamma^{-1}(t))^2)},
	\end{multline}
	and we get the same way that
	\begin{multline} \label{522}
		K_{\gamma,r,\varphi} \left(f,t^r,L^p_{w_{\alpha}}\left( \textstyle \frac{A_2'}{(\gamma^{-1}(t))^2},\infty \right),W_{\gamma}^{r,p}\left( \textstyle \frac{A_2'}{(\gamma^{-1}(t))^2},\infty \right) \right) \\
		\leq \inf_{q \in P_{r-1}} \| (f - q \circ \gamma) w_{\alpha} \|_{L^p \left( \frac{A_2'}{(\gamma^{-1}(t))^2},\infty \right)}.
	\end{multline}
	By the inequalities (\ref{487}), (\ref{520}), (\ref{521}) and (\ref{522}) we get the inequality (\ref{482}), which we wanted to prove.
	Since in the resulting estimation $C$ is independent of $p$, the statement is true for $p = \infty$ too.
	\end{lepess}
	
	\begin{lepess} \label{581}
		 We get a lower bound for $K_{\gamma,r,\varphi}(f,t^r,L^p_{w_{\alpha}},W_{\gamma}^{r,p})$. Let $A_1,A_2$ be constants as in the second section in the proof of theorem \ref{50} (when we proved (\ref{391})), and let $0 < t \leq T$ be such that $T$ satisfy the assumption (\ref{358a}) and the following:
	\begin{equation*} 
		T < \min \left\{ \textstyle \gamma \left( \sqrt{\frac{a_1}{8 A_1 r^2}} \right), \gamma \left( \sqrt{\frac{A_2}{a_N + 1}} \right), \gamma \left( \sqrt{\frac{A_2}{2 \alpha}} \right) \displaystyle \right\}.
	\end{equation*}
	We must keep the last condition only if $\alpha > 0$. By definition (\ref{480}) of $\omega_{\gamma,\varphi}^{r}(f,t)_{w_{\alpha,p}}$ it is sufficient to prove that all terms that determine the complete modulus of smoothness can be estimated with a constant-fold of the $\gamma$-$K$-functional. By the second part of theorem \ref{50}:
	\begin{equation} \label{524}
		\Omega_{\gamma,\varphi}^{r}(f,t)_{w_{\alpha},p} \leq C \cdot \widetilde{K}_{\gamma,r,\varphi}(f,t^r,L^p_{w_{\alpha}},W_{\gamma}^{r,p}) \leq C K_{\gamma,r,\varphi}(f,t^r,L^p_{w_{\alpha}},W_{\gamma}^{r,p}),
	\end{equation}
	so it is sufficient to show that
	\begin{equation} \label{525}
		\inf_{p \in P_{r-1}} \| w_{\alpha} (f - p \circ \gamma) \|_{L^p(0,4A_1r^2 (\gamma^{-1}(t))^2)} \leq C K_{\gamma,r,\varphi}(f,t^r,L^p_{w_{\alpha}},W_{\gamma}^{r,p})
	\end{equation}
	and
	\begin{equation} \label{526}
		\inf_{q \in P_{r-1}} \| w_{\alpha} (f - q \circ \gamma) \|_{L^p\left( \frac{A_2}{(\gamma^{-1}(t))^2}, \infty \right)} \leq C K_{\gamma,r,\varphi}(f,t^r,L^p_{w_{\alpha}},W_{\gamma}^{r,p}).
	\end{equation}
	Let $1 < p < \infty$, $g \in W_{\gamma}^{r,p}$ be an arbitrary function. Then by statement \ref{17} the Taylor polynomial of degree $r-1$ related to the function $g \circ \gamma^{-1}$ with starting point $\gamma(4A_1 r^2 (\gamma^{-1}(t))^2)$ is well defined, let it be $\widetilde{T}_{r-1}$. Then
	\begin{multline} \label{527}
		\inf_{p \in P_{r-1}} \| w_{\alpha} (f - p \circ \gamma) \|_{L^p(0,4A_1r^2 (\gamma^{-1}(t))^2)} \leq \| w_{\alpha} (f - \widetilde{T}_{r-1} \circ \gamma) \|_{L^p(0,4A_1r^2 (\gamma^{-1}(t))^2)} \\
		\leq \| (f-g) w_{\alpha} \|_{L^p(0,4A_1r^2 (\gamma^{-1}(t))^2)} + \| (g - \widetilde{T}_{r-1} \circ \gamma) w_{\alpha} \|_{L^p(0,4A_1r^2 (\gamma^{-1}(t))^2)}.
	\end{multline}
	We can get with the help of the statement \ref{17} that $\widetilde{T}_{r-1} \circ \gamma(x) = \widetilde{T}_{\gamma,r-1}(x)$,
	where $\widetilde{T}_{\gamma,r-1}$ is the generalized Taylor polynomial of degree $r-1$ related to $g$ with starting point $4A_1 r^2 (\gamma^{-1}(t))^2$, so by the statement \ref{31} we get
	\begin{multline} \label{530}
		\| (g - \widetilde{T}_{r-1} \circ \gamma) w_{\alpha} \|_{L^p(0,4A_1r^2 (\gamma^{-1}(t))^2)} \\
		\leq C \left( \int_{0}^{4A_1r^2 (\gamma^{-1}(t))^2} \left| \int_{0}^{4A_1r^2 (\gamma^{-1}(t))^2} g_{\gamma}^{(r)}(u) (\gamma(u) - \gamma(x))_+^{r-1} \, \mathrm{d} \gamma(u) \right|^p w_{\alpha}^p(x) \, \mathrm{d}x  \right)^{\frac{1}{p}}.
	\end{multline}
	By using the Minkowski integral inequality and the Hölder inequality:
	\begin{multline} \label{542}
		\| (g - \widetilde{T}_{r-1} \circ \gamma) w_{\alpha} \|_{L^p(0,4A_1r^2 (\gamma^{-1}(t))^2)} \\ 
		\leq C \underbrace{\left( \int_{0}^{4A_1r^2 (\gamma^{-1}(t))^2} \left| g_{\gamma}^{(r)}(u) w_{\alpha}(u) \varphi^r(u) \right|^p \gamma'(u) \, \mathrm{d} u \right)^{\frac{1}{p}}}_{I_1} \\
		\times \underbrace{\left( \int_{0}^{4A_1r^2 (\gamma^{-1}(t))^2} \frac{1}{{w_{\alpha}^q(u)} \varphi^{qr}(u)} \left( \int_{0}^{u} (\gamma(u) - \gamma(x))^{p(r-1)} w_{\alpha}^p(x) \, \mathrm{d} x \right)^{\frac{q}{p}} \gamma'(u) \, \mathrm{d} u \right)^{\frac{1}{q}}}_{I_2},
	\end{multline}
	where $\frac{1}{p} + \frac{1}{q} = 1$. The $\mathrm{d} \gamma(u) = \gamma'(u) \, \mathrm{d}u$ transformation is correct, because we chose $t$ such that $u \leq 4A_1r^2 (\gamma^{-1}(t))^2 < \frac{a_1}{2}$, so $\gamma'(u) \leq C$.
	Because of this the first factor of the given product can be estimated as the following:
	\begin{equation} \label{544}
		I_1 \leq \| g_{\gamma}^{(r)} w_{\alpha} \varphi^r \|_{L^p(0,4A_1r^2 (\gamma^{-1}(t))^2)}.
	\end{equation}
	In the second factor, since $\alpha \geq 0$, thus $w_{\alpha}(x) \leq C w_{\alpha}(u)$,
	because by the properly choosing of $t$ we have $e^{u} \leq e^{4A_1r^2 (\gamma^{-1}(t))^2} \leq e^{\frac{a_1}{2}}$. So
	\begin{multline*} 
		\frac{1}{{w_{\alpha}^q(u)} \varphi^{qr}(u)} \left( \int_{0}^{u} (\gamma(u) - \gamma(x))^{p(r-1)} w_{\alpha}^p(x) \, \mathrm{d} x \right)^{\frac{q}{p}} \\
		\leq \frac{C^q}{\varphi^{qr}(u)} \left( \int_{0}^{\gamma(u)} (\gamma(u) - y)^{p(r-1)} \, \mathrm{d} \gamma^{-1}(y) \right)^{\frac{q}{p}} \\
		\leq \frac{C^q}{\varphi^{qr}(u)} \left( \int_{0}^{\gamma(u)} (\gamma(u) - y)^{p(r-1)} \, \mathrm{d} y \right)^{\frac{q}{p}} \leq \frac{C^q}{\varphi^{qr}(u)} (\gamma(u))^{(p(r-1) + 1) \frac{q}{p}},
	\end{multline*}
	where we used again that $0 < y < \gamma(u) \leq \gamma \left( \frac{a_1}{2} \right)$, and $(\gamma^{-1})'$ is bounded on $\left[ 0, \gamma \left( \frac{a_1}{2} \right) \right]$.
	
	Since $u < \frac{a_1}{2} < a_1$, by (\ref{198}) there is a constant $C$ such that
	\begin{equation*} 
		\frac{C^q}{\varphi^{qr}(u)} (\gamma(u))^{(p(r-1) + 1) \frac{q}{p}} \leq  C^q u^{\frac{qr}{2} - 1}.
	\end{equation*}
	So since $\gamma'(u) \leq C$, we get for the second factor that
	\begin{equation} \label{548}
		I_2 \leq C (\gamma^{-1}(t))^{\frac{r}{2}} \leq C t^r,
	\end{equation}
	because $t < \gamma(a_1)$ and (\ref{200}) is valid. We got by (\ref{530}), (\ref{542}), (\ref{544}) and (\ref{548}) that
	\begin{equation*} 
		\| (g - \widetilde{T}_{r-1} \circ \gamma) w_{\alpha} \|_{L^p(0,4A_1r^2 (\gamma^{-1}(t))^2)} \leq C t^r \| w_{\alpha} \varphi^r g^{(r)}_{\gamma} \|_{L^p(0,4A_1r^2 (\gamma^{-1}(t))^2)},
	\end{equation*}
	so by (\ref{527})
	\begin{equation*} 
		\inf_{p \in P_{r-1}} \| w_{\alpha} (f - p \circ \gamma) \|_{L^p(0,4A_1r^2 (\gamma^{-1}(t))^2)} 
		\leq C \left(  \| (f-g) w_{\alpha} \|_{L^p} + t^r \| w_{\alpha} \varphi^r g^{(r)}_{\gamma} \|_{L^p} \right).
	\end{equation*}
	Since this is true for all $g \in W_{\gamma}^{r,p}$ functions, we proved (\ref{525}).
	
	We do similarly at the proof of (\ref{526}): let $1 < p < \infty$, $g \in W_{\gamma}^{r,p}$ be an arbitrary function, and $T_{r-1}$ is the Taylor polynomial of degree $r-1$ related to the function $g \circ \gamma^{-1}$ with starting point $\gamma \left( \frac{A_2}{(\gamma^{-1}(t))^2} \right)$. Then
	\begin{multline} \label{533}
		\inf_{q \in P_{r-1}} \| w_{\alpha} (f - q \circ \gamma) \|_{L^p\left( \frac{A_2}{(\gamma^{-1}(t))^2}, \infty \right)} \leq \| w_{\alpha} (f - T_{r-1} \circ \gamma) \|_{L^p\left( \frac{A_2}{(\gamma^{-1}(t))^2}, \infty \right)} \\
		\leq \| (f-g) w_{\alpha} \|_{L^p\left( \frac{A_2}{(\gamma^{-1}(t))^2}, \infty \right)} + \| (g - T_{r-1} \circ \gamma) w_{\alpha} \|_{L^p\left( \frac{A_2}{(\gamma^{-1}(t))^2}, \infty \right)}.
	\end{multline}
	By the statement \ref{17} we get $T_{r-1} \circ \gamma = T_{\gamma,r-1}$, where this latter function is the generalized Taylor polynomial of degree $r-1$ related to $g$ with starting point $\frac{A_2}{(\gamma^{-1}(t))^2}$. We get by the statement \ref{31} that
	\begin{equation} \label{534}
		g(x) - T_{r-1} \circ \gamma (x) = \mathcal{C}_1^r \int_{\frac{A_2}{(\gamma^{-1}(t))^2}}^{x} g_{\gamma}^{(r)}(u) \frac{(x - u)^{r-1}}{(r-1)!} \, \mathrm{d}u,
	\end{equation}
	and due to the proper choice of $t$ we have $\frac{A_2}{(\gamma^{-1}(t))^2} > a_N + 1$, and if $a_{N}+1 < u$, then $\gamma(u) = \mathcal{C}_1 u + \mathcal{C}_2$, so $\gamma(x) - \gamma(u) = \mathcal{C}_1(x-u)$ and $\mathrm{d} \gamma(u) = \mathcal{C}_1 \, \mathrm{d}u$.
	
	If $\alpha > 0$, then by the good choice for $t$ we have $2 \alpha < \frac{A_2}{(\gamma^{-1}(t))^2} < u < x$, thus the function $x \mapsto e^{-\frac{x}{2}} x^{\alpha}$ is strictly decreasing, and if $\alpha = 0$, then the strictly increasing property is automatically fulfilled. So $w_{\alpha}(x) \leq e^{- \frac{x}{2}} e^{\frac{u}{2}} w_{\alpha}(u)$,
	and by (\ref{534}) and the Hölder inequality we have
	\begin{multline} \label{536}
		|(g(x) - T_{r-1} \circ \gamma(x)) w_{\alpha}(x)| \leq C e^{-\frac{x}{2}} \int_{\frac{A_2}{(\gamma^{-1}(t))^2}}^{x} \left( e^{\frac{u}{2}} (x-u)^{r-1} \right)^{\frac{1}{p} + \frac{1}{q}} |g^{(r)}_{\gamma}(u) w_{\alpha}(u)| \, \mathrm{d}u \\
		\leq C e^{-\frac{x}{2}} \left( \int_{\frac{A_2}{(\gamma^{-1}(t))^2}}^{x} e^{\frac{u}{2}} (x-u)^{r-1} \, \mathrm{d}u \right)^{\frac{1}{q}} \left( \int_{\frac{A_2}{(\gamma^{-1}(t))^2}}^{x} e^{\frac{u}{2}} (x-u)^{r-1} |g^{(r)}_{\gamma}(u) w_{\alpha}(u)|^p \, \mathrm{d}u \right)^{\frac{1}{p}}.
	\end{multline}
	Since
	\begin{equation*} 
		\int_{\frac{A_2}{(\gamma^{-1}(t))^2}}^{x} e^{\frac{u}{2}} (x-u)^{r-1} \, \mathrm{d}u = e^{\frac{x}{2}} 2^r (r-1)!,
	\end{equation*}
	we can estimate (\ref{536}):
	\begin{equation} \label{538}
		|(g(x) - T_{r-1} \circ \gamma(x)) w_{\alpha}(x)| \leq C e^{-\frac{x}{2p}} \left( \int_{\frac{A_2}{(\gamma^{-1}(t))^2}}^{x} e^{\frac{u}{2}} (x-u)^{r-1} |g^{(r)}_{\gamma}(u) w_{\alpha}(u)|^p \, \mathrm{d}u \right)^{\frac{1}{p}}.
	\end{equation}
	On the integration domain $\frac{A_2}{(\gamma^{-1}(t))^2} < u$ holds, by using that $t < \gamma(a_1)$ and (\ref{200}) there is a constant $C$ such that
	\begin{equation} \label{540}
		\textstyle 1 \leq \frac{\gamma^{-1}(t) \sqrt{u}}{\sqrt{A_2}} \leq C t \sqrt{u}.
	\end{equation}
	By (\ref{538}) and (\ref{540}), with the Fubini theorem we get
	\begin{multline*} 
		\| (g - T_{r-1} \circ \gamma) w_{\alpha} \|_{L^p\left( \frac{A_2}{(\gamma^{-1}(t))^2}, \infty \right)} \\
		\leq C t^r \left( \int_{\frac{A_2}{(\gamma^{-1}(t))^2}}^{\infty} |g^{(r)}_{\gamma}(u) \varphi^r(u) w_{\alpha}(u)|^p \int_{u}^{\infty} e^{- \frac{x-u}{2}} (x-u)^{r-1} \, \mathrm{d}x \mathrm{d}u \right)^{\frac{1}{p}} \\
		\leq C t^r \| g^{(r)}_{\gamma} \varphi^r w_{\alpha} \|_{L^p\left( \frac{A_2}{(\gamma^{-1}(t))^2},\infty \right)},
	\end{multline*}
	because the inner integral is finite.
	\end{lepess}
	
	Finally by (\ref{533}):
	\begin{equation*} 
		\inf_{q \in P_{r-1}} \| w_{\alpha} (f - q \circ \gamma) \|_{L^p\left( \frac{A_2}{(\gamma^{-1}(t))^2}, \infty \right)} 
		\leq C \left( \| (f-g) w_{\alpha} \|_{L^p} + t^r \| g^{(r)}_{\gamma} \varphi^r w_{\alpha} \|_{L^p} \right).
	\end{equation*}
	Since this estimation holds for all $g \in W_{\gamma}^{r,p}$ functions, we proved (\ref{526}) too, and by (\ref{524}), (\ref{525}) and (\ref{526}), with the definition of the complete $\gamma$-modulus of smoothness we get exactly (\ref{483}). Since the constats in the estimations were independent of $p$, the statement is true for $p = \infty$ too. For $p=1$ we can repeat the steps above word by word, the estimates hold also without the application of the Hölder inequality, so we get again (\ref{483}) for $p=1$.
\end{proof}

\section{Acknowledgements}

The author wishes to express his thanks to \'Agota P. Horváth for her many helpful suggestions and useful remarks in writing this
paper. Furthermore, many thanks to Andor Szab\'o for the grammatic verifying.


\bibliography{altalanos_hivatkozas3}

\end{document}